\documentclass[12pt]{amsart}

\usepackage{amsthm}
\usepackage{amsmath}

\usepackage{amssymb}

\usepackage{enumerate}

\usepackage{graphicx}

\makeatletter
\@namedef{subjclassname@2010}{%
  \textup{2010} Mathematics Subject Classification}
\makeatother

\newtheorem{thm}{Theorem}[section]
\newtheorem{cor}[thm]{Corollary}
\newtheorem{lem}[thm]{Lemma}
\newtheorem{prop}[thm]{Proposition}

\newtheorem*{thm.2.6}{Theorem 2.6}
\newtheorem*{thm.2.7}{Theorem 2.7}
\newtheorem*{thm.4.4}{Theorem 4.4}
\newtheorem*{thm.5.3}{Theorem 5.3}
\newtheorem*{thm.6.3}{Theorem 6.3}
\newtheorem*{thm.6.4}{Theorem 6.4}

\theoremstyle{definition}
\newtheorem{defin}[thm]{Definition}
\newtheorem{rem}[thm]{Remark}

\newtheorem*{claim}{Claim}
\newtheorem*{claim1}{Claim 1}
\newtheorem*{claim2}{Claim 2}
\newtheorem*{claim3}{Claim 3}
\newtheorem*{claim4}{Claim 4}
\newtheorem*{claim5}{Claim 5}

\newtheorem{problem}[thm]{Problem}
\newtheorem*{notn}{Notation}

\numberwithin{equation}{section}

\newcommand{\Fin}{\mathrm{Fin}}

\newcommand{\al}{\alpha}
\newcommand{\om}{\omega}

\newcommand{\sse}{\subseteq}
\newcommand{\contains}{\supseteq}

\newcommand{\FIN}{\mathrm{FIN}}

\newcommand{\rgl}{\rangle}
\newcommand{\lgl}{\langle}

\newcommand{\re}{\restriction}

\newcommand{\dom}{\mathrm{dom}\,}

\newcommand{\ra}{\rightarrow}

\newcommand{\Llra}{\Longleftrightarrow}

\newcommand{\rank}{\mathrm{rank}}

\begin{document}


\baselineskip=17pt



\title[Canonical cofinal maps]{Continuous and other finitely generated canonical cofinal maps on ultrafilters}

\author[N. Dobrinen]{Natasha Dobrinen}
\address{Department of Mathematics\\
 University of Denver \\
C.M.\ Knudson Hall, Room 300\\
2390 S.\ York St.\\ Denver, CO \ 80208 U.S.A.}
\email{natasha.dobrinen@du.edu}
\urladdr{http://cs.du.edu/~ndobrine}
\date{}

\thanks{This work was partially supported by National Science Foundation Grants DMS-142470 and DMS-1600781,
Simons Foundation Collaboration Grant  245286, and a University of Denver Faculty Research Grant}

\begin{abstract}
This paper investigates conditions under which canonical cofinal maps of the  following three types exist: continuous,   generated by finitary  end-extension preserving  maps, and generated by  finitary maps.
The  main theorems  prove that every monotone cofinal map on an ultrafilter from a certain class of ultrafilters is actually canonical when restricted to some cofinal subset.
These theorems are then applied to find connections between Tukey, Rudin-Keisler, and Rudin-Blass reducibilities on large classes of ultrafilters.

The main theorems on canonical cofinal maps are the following.
Under a mild assumption, basic Tukey reductions are inherited under  Tukey reduction.
In particular, every ultrafilter Tukey reducible to a p-point has continuous Tukey reductions.
If $\mathcal{U}$ is a  Fubini iterate of p-points, then  each monotone cofinal map  from $\mathcal{U}$ to some other ultrafilter is generated (on a cofinal subset of $\mathcal{U}$) by a finitary map on the base tree for $\mathcal{U}$ which is monotone and end-extension preserving - the analogue of continuous in this context.
Further, every ultrafilter which is Tukey reducible to some Fubini iterate of  p-points has finitely generated cofinal maps.
Similar theorems also hold for some other classes of ultrafilters.
\end{abstract}



\maketitle

\section{Introduction}\label{sec.intro}

In this paper, ultrafilters on countable base sets are considered to be partially ordered by reverse inclusion.
A map from an ultrafilter $\mathcal{U}$ to another ultrafilter $\mathcal{V}$ is {\em cofinal} if every image of a cofinal subset  of $\mathcal{U}$ is a  cofinal subset of  $\mathcal{V}$.
We say that $\mathcal{V}$ is {\em Tukey reducible} to $\mathcal{U}$ and write $\mathcal{V}\le_T\mathcal{U}$ if and only if there is a cofinal map from $\mathcal{U}$ to $\mathcal{V}$.
When $\mathcal{U}\le_T\mathcal{V}$ and
$\mathcal{V}\le_T\mathcal{U}$, then we say that $\mathcal{U}$ is {\em Tukey equivalent to} $\mathcal{V}$ and write $\mathcal{U}\equiv_T\mathcal{V}$.
It is clear that $\equiv_T$ is an equivalence relation, and $\le_T$ on the equivalence classes forms a partial ordering.
The equivalence classes are called {\em Tukey types}.
We point out  that  since $\contains$ is a directed partial ordering on an ultrafilter,
two ultrafilters are Tukey equivalent if and only if they are  {\em cofinally similar};  that is, there is a partial ordering  into which they both embed as cofinal subsets (see \cite{Tukey40}).
Thus, for ultrafilters, Tukey equivalence is the same as cofinal similarity.
An equivalent formulation of Tukey reducibility, noticed by Schmidt in \cite{Schmidt55},
shows that $\mathcal{V}\le_T\mathcal{U}$ if and only if there is a {\em Tukey map} from $\mathcal{V}$ to $\mathcal{U}$; that is a map $g:\mathcal{V}\ra\mathcal{U}$ such that every unbounded (with respect to the partial ordering $\contains$) subset of $\mathcal{V}$ is unbounded in $\mathcal{U}$.

This paper focuses on the existence of canonical cofinal maps of three types: continuous,  approximated by {\em basic} maps (end-extension and level  preserving finitary maps - see Definitions \ref{defn.basic}  and \ref{defn.basicTredtree}), and approximated by monotone finitary maps (see Definition \ref{defn.finitaryTred}).
In each of these cases, the original cofinal map is generated by the approximating finitary map.

The notion of  Tukey reducibility between two  directed partial orderings was first introduced by Tukey in
 \cite{Tukey40}
to study the Moore-Smith theory of net convergence in topology.
This naturally led to investigations of Tukey types of more general partial orderings, directed and later non-directed.
These investigations
 often reveal useful information for the comparison of different partial orderings.
For example, Tukey reducibility  preserves calibre-like properties, such as the countable chain condition, property K, precalibre $\aleph_1$, $\sigma$-linked, and $\sigma$-centered (see \cite{Todorcevic96}).
For more on classification theories of Tukey types for certain classes of ordered sets, we refer the reader to \cite{Tukey40},
\cite{Day44}, \cite{Isbell65}, \cite{TodorcevicDirSets85}, and
\cite{Todorcevic96}.
As the  focus of this paper is canonical cofinal maps on ultrafilters, and as we have recently written a survey article giving an overview of the   motivation and  the state of the art of the Tukey theory of ultrafilters
(see \cite{DobrinenTukeySurvey15}),
we   present here only the background and motivations relevant for this work.

For ultrafilters, we may restrict our attention to monotone cofinal maps.
A map $f:\mathcal{U}\ra\mathcal{V}$ is {\em monotone}  if
 for any $X,Y\in\mathcal{U}$, $X\contains Y$ implies $f(X)\contains f(Y)$.
It is not hard to show that whenever $\mathcal{U}\ge_T\mathcal{V}$, then there is a \em monotone \rm cofinal map witnessing this (see Fact 6 of \cite{Dobrinen/Todorcevic11}).

As cofinal maps between ultrafilters have domain and range of size continuum,
a priori, the Tukey type of an ultrafilter may have size $2^{\mathfrak{c}}$.
Indeed, this is the case for ultrafilters which have the maximum Tukey type $([\mathfrak{c}]^{<\om},\sse)$.
However, if an ultrafilter has the property that every Tukey reduction from it to another ultrafilter may be witnessed by a continuous map,
then it follows that its Tukey type, as well as the Tukey type of each ultrafilter Tukey reducible to it, has size at most continuum.
This is the case for p-points.

\begin{defin}\label{defn.p-point}
An ultrafilter $\mathcal{U}$ on $\om$ is a \em p-point \rm iff for each decreasing sequence $X_0\contains X_1\contains \dots$ of elements of $\mathcal{U}$, there is an $U\in\mathcal{U}$
such that $U\sse^* X_n$, for all $n<\om$.
\end{defin}

\begin{defin}\label{defn.ctsandfinitary}
An ultrafilter   $\mathcal{U}$ on $\om$
 {\em has continuous Tukey reductions}  if whenever  $f:\mathcal{U}\ra\mathcal{V}$ is a monotone cofinal map, there is a cofinal subset $\mathcal{C}\sse\mathcal{U}$ such that $f\re\mathcal{C}$ is continuous.
\end{defin}

The following Theorem 20 in \cite{Dobrinen/Todorcevic11} has provided a fundamental tool for all subsequent   research on the classification of Tukey types of p-points.

\begin{thm}[Dobrinen/Todorcevic \cite{Dobrinen/Todorcevic11}]\label{thm.ppoint}
Suppose $\mathcal{U}$ is a p-point on $\om$.
 Given a monotone cofinal map $f$ from $\mathcal{U}$ into another ultrafilter,
 there is an $X\in \mathcal{U}$ such that
 $f$ is continuous on $\mathcal{C}=\{Y\in\mathcal{U}:Y\sse X\}$.
In particular, $\mathcal{U}$ has continuous Tukey reductions.
Moreover, these continuous Tukey reductions are generated by  monotone basic  functions.
\end{thm}

\begin{rem}\label{rem.cts}
The proof of Theorem 20 in \cite{Dobrinen/Todorcevic11}  shows that p-points have the stronger property of monotone basic Tukey reductions (see Definition \ref{defn.basic}).
\end{rem}

It was later proved by   Raghavan\footnote{In this joint paper of Raghavan and Todorcevic, it is clearly stated which results in the paper are due to each of the authors.  In accordance with the wishes of the second author expressly stated to the author of this paper in 2011,
whenever citing results from \cite{Raghavan/Todorcevic12}, we clearly state to whom the result is due.}\  in \cite{Raghavan/Todorcevic12} that any ultrafilter Tukey reducible to some basically generated ultrafilter
has Tukey type of cardinality at most $\mathfrak{c}$.
The notion of a basically generated ultrafilter first appeared as Definition 15 in \cite{Dobrinen/Todorcevic11}.
The
 slightly modified version from \cite{Raghavan/Todorcevic12}
 will be used in this paper, as it will simplify  statements.
In this paper, by  a {\em filter base} for an ultrafilter $\mathcal{U}$, we mean a cofinal subset of $\mathcal{U}$ which is closed under finite intersections.

\begin{defin}[Definition 13, \cite{Raghavan/Todorcevic12}]\label{def.bg}
An ultrafilter $\mathcal{U}$ on $\om$ is {\em basically generated} if it has a filter base $\mathcal{B}\sse\mathcal{U}$ such that each sequence $\lgl A_n:n<\om\rgl$ of members of $\mathcal{B}$ converging to another member of $\mathcal{B}$ has a subsequence whose intersection is in $\mathcal{U}$.
\end{defin}

It was shown in  \cite{Dobrinen/Todorcevic11} that the class of basically generated ultrafilters contains all p-points and is closed under taking Fubini products.
It is still unknown whether the class of all Fubini iterates of p-points is the same as or strictly contained in the class of all basically generated ultrafilters.

Continuous cofinal maps provide one of the main  keys to  the analysis of
the structure of the Tukey types of p-points  (see for instance \cite{Dobrinen/Todorcevic11}, \cite{Raghavan/Todorcevic12}  and \cite{Raghavan/Shelah17}).
Moreover, continuous cofinal maps are   crucial to
providing a mechanism for applying Ramsey-classification theorems
 on barriers to
classify the initial Tukey structures and Rudin-Keisler structures within these for a large class of p-points:
selective ultrafilters in
 \cite{Raghavan/Todorcevic12};
weakly Ramsey ultrafilters and a large hierarchy of
 rapid p-points satisfying partition relations in
 \cite{Dobrinen/Todorcevic14} and \cite{Dobrinen/Todorcevic15};
 and $k$-arrow ultrafilters, hypercube ultrafilters, and a large class of p-points constructed using products of  Fra\"{i}ss\'{e}\ classes satisfying the Ramsey property in
\cite{Dobrinen/Mijares/Trujillo14}.
Similar methods were developed for a hierarchy of non-p-points above a Ramsey ultrafilter in \cite{DobrinenJSL15}.

Continuous cofinal maps are also used in
the following theorem, which reveals the surprising fact that the Tukey and Rudin-Blass orders sometimes coincide.
Recall that $\mathcal{V}\le_{RB}\mathcal{U}$ if and only if there is  a finite-to-one map $f:\om\ra\om$ such that $\mathcal{V}=f(\mathcal{U})$, where $f(\mathcal{U})$ is defined to be the set of all $X\sse\om$ such that $f^{-1}[X]$ is a member of $\mathcal{U}$.
The following is Theorem 10 in \cite{Raghavan/Todorcevic12}.

\begin{thm}[Raghavan \cite{Raghavan/Todorcevic12}]\label{thm.RT}
Let $\mathcal{U}$ be any ultrafilter and let $\mathcal{V}$ be a q-point.
If $\mathcal{V}\le_T\mathcal{U}$ and this is witnessed by
 a continuous, monotone cofinal map from $\mathcal{U}$ to $\mathcal{V}$,
then  $\mathcal{V}\le_{RB}\mathcal{U}$.
\end{thm}

In Section \ref{sec.pres},
we  prove in Theorem \ref{thm.PropB} that, under a mild assumption, the property of having basic cofinal maps is inherited under Tukey reduction.
The proof uses the Extension Lemma \ref{thm.PropACtsMaps} showing that   any basic monotone map on a cofinal subset of an ultrafilter may be extended to a basic monotone map on all of $\mathcal{P}(\om)$.
In particular, p-points satisfy the mild assumption;
hence we obtain the following theorem.

\begin{thm.2.6}
Every ultrafilter Tukey reducible to a  p-point
 has basic, and hence continuous, Tukey reductions.
\end{thm.2.6}

 Combined, Theorems \ref{thm.RT} and \ref{thm.main} imply the following.

\begin{thm.2.7}
If $\mathcal{U}$ is Tukey reducible to some p-point, then any q-point Tukey below $\mathcal{U}$ is actually Rudin-Blass below $\mathcal{U}$.
\end{thm.2.7}

The rest of the paper involves finding the analogues of Theorems \ref{thm.ppoint} and 2.6
for countable iterations of Fubini products of p-points and applying them to connect Tukey reduction with Rudin-Keisler and Rudin-Blass reductions.
We now delineate these results.

Section \ref{sec.Fubit}
is a primer, explicitly showing how any countable iteration of Fubini products of p-points, which we also simply call  a {\em Fubini iterate of p-points}, can be viewed as an ultrafilter generated by trees on a  front on $\om$.
This precise way of viewing Fubini iterates of p-points sets the stage for finding the analogue of Theorem \ref{thm.ppoint} for this more general class of ultrafilters.
It is not possible to show that Fubini iterates of p-points have continuous Tukey reductions.
In particular, 
Corollary 11 of Raghavan  in \cite{Raghavan/Todorcevic12} 
shows that the Fubini product of two 
 non-isomorphic selective ultrafilters does not have a continuous cofinal map into one of the selective ultrafilters. 
However, 
we will  show that  key properties of continuous maps still hold for this class of ultrafilters.

In Section \ref{sec.BasicCofinalFubit}
we define  the notion of  a {\em basic map} for Fubini iterates, which is in particular an end-extension preserving  map from finite subsets of  the tree $\hat{B}$ of initial segments  of  members of   a  front $B$ into finite subsets of $\om$ (see
Definition \ref{defn.basicTredtree}).
This is the analogue of continuity for Fubini iterates of p-points.
One of the main results of this paper is the following.

 \begin{thm.4.4} Fubini iterates of p-points have basic Tukey reductions.
\end{thm.4.4}

Thus, monotone cofinal maps on Fubini iterates of p-points are continuous, with respect to the product topology on the space  $2^{\hat{B}}$.
As basic maps on  fronts have the key property (end-extension preserving) of continuous maps used to convert Tukey reduction to Rudin-Keisler reduction in
\cite{Raghavan/Todorcevic12},
\cite{Dobrinen/Todorcevic14}, \cite{Dobrinen/Todorcevic15},  \cite{Dobrinen/Mijares/Trujillo14}, and \cite{DobrinenJSL15}, it seems likely that they will play a crucial role in obtaining similar results for ultra-Ramsey spaces of Chapter 6 of Todorcevic's book \cite{TodorcevicBK10}.

Sections \ref{sec.directapp} and \ref{sec.finitegen} contain applications of Theorem \ref{thm.allFubProd_p-point_cts} to a broad class of ultrafilters.
In Section \ref{sec.directapp}, we directly apply Theorem   \ref{thm.allFubProd_p-point_cts}  to obtain an analogue of  Theorem 10 of Raghavan in \cite{Raghavan/Todorcevic12}.
In Theorem \ref{thm.RBonBhat}, we prove that if $\mathcal{U}$ is a Fubini iterate of p-points and $\mathcal{V}$ is a q-point Tukey reducible to $\mathcal{U}$,
then there is a finite-to-one map on a large subset of $\hat{B}$, where $B$ is the  front base for $\mathcal{U}$, such that its image on $\mathcal{U}$ generates a subfilter of $\mathcal{V}$.
One of the consequences of this is  the following.

\begin{thm.5.3}
Suppose $\mathcal{U}$ is a finite iteration of Fubini products of p-points.
If $\mathcal{V}$ is a q-point and
 $\mathcal{V}\le_T\mathcal{U}$, then
$\mathcal{V}\le_{RK}\mathcal{U}$.
\end{thm.5.3}

This improves one aspect of Corollary 56  of Raghavan in \cite{Raghavan/Todorcevic12} as  $\mathcal{V}$ is only required to be a q-point, not a selective ultrafilter.
The improvement though comes at the expense of
limiting $\mathcal{U}$ to a finite Fubini iterate of p-points.  It is unknown whether this can be extended to all Fubini iterates of p-points.

In Section \ref{sec.finitegen}
we prove the analogue of Theorem  \ref{thm.PropB}
for ultrafilters Tukey reducible to a Fubini iterate of p-points.
(See Definition \ref{defn.finitaryTred} for
{\em finitely generated Tukey reductions}.)

\begin{thm.6.3}\label{thm.FinitelyRepTukeyRed}
Let $\mathcal{U}$ be any Fubini iterate of  p-points.
If $\mathcal{V}\le_T\mathcal{U}$, then
$\mathcal{V}$  has finitely generated Tukey reductions.
\end{thm.6.3}

These finitary maps are an improvement on the maps $\psi_{\varphi}$ used  in \cite{Raghavan/Todorcevic12} (see Definition 7 in  \cite{Raghavan/Todorcevic12})
in the sense that our finitary maps are shown to  generate the original cofinal maps.
Theorem \ref{thm.FinitelyRepTukeyRed} is used to extend Theorem 17 of Raghavan in \cite{Raghavan/Todorcevic12} to the class of all ultrafilters Tukey reducible to some Fubini iterate of p-points, in contrast to his result where $\mathcal{U}$ is assumed to be basically generated.
It is still open whether every ultrafilter Tukey reducible to a
p-point is basically generated (see discussion around Problem \ref{prob.bgTdownwardinherit}), as the class of  basically generated ultrafilters and the class of ultrafilters Tukey reducible to a Fubini iterate of p-points may be very different.

\begin{thm.6.4}\label{thm.rextension}
If $\mathcal{U}$ is Tukey reducible to a   Fubini iterate of p-points, then
for each $\mathcal{V}\le_T\mathcal{U}$, there is a filter $\mathcal{U}(P)\equiv_T\mathcal{U}$ such that $\mathcal{V}\le_{RK}\mathcal{U}(P)$.
\end{thm.6.4}

The paper closes with a list of open problems in Section \ref{sec.op}.

The results in Sections \ref{sec.pres}, \ref{sec.Fubit} and \ref{sec.BasicCofinalFubit} were completed in 2010, presented  at the Logic Colloquium in Paris that year, and have appeared in  the unpublished preprint \cite{Dobrinen10}.
The present paper includes much revised presentations and proofs  of those results, new extensions of them,  and additional applications.

\bf Acknowledgement. \rm 
Profuse thanks go to the referee, whose thorough reading and suggestions greatly improved this paper, especially the presentation in Section 4.


\section{Basic Tukey reductions  inherited under Tukey reducibility}\label{sec.pres}

One of the crucial  tools
used to determine the  structure of the Tukey types of p-points is the existence of continuous cofinal maps (Theorem 20 in \cite{Dobrinen/Todorcevic11}).
Continuity contributes to the analysis of the structure of the Tukey types of p-points by essentially reducing the number of cofinal maps under consideration from $2^{\mathfrak{c}}$ to $\mathfrak{c}$, with the immediate consequence that there are at most $\mathfrak{c}$ many ultrafilters Tukey reducible to any p-point.
Continuity further contributes to finding exact Tukey and Rudin-Keisler structures below certain classes of p-points satisfying partition relations.
The fact that each monotone cofinal map on a p-point  is  approximated by a finitary  end-extension preserving function
 is what
allows for  application of Ramsey-classification theorems to find
the exact Tukey and Rudin-Keisler structures below the p-points  forced by certain topological Ramsey spaces (see \cite{Raghavan/Todorcevic12},
\cite{Dobrinen/Todorcevic14}, \cite{Dobrinen/Todorcevic15}, and
\cite{Dobrinen/Mijares/Trujillo14}).
Other
 applications of cofinal maps represented by finitary end-extension preserving maps
 appear in Dobrinen's  contributions in \cite{Blass/Dobrinen/Raghavan15},  and further
  in \cite{DobrinenJSL15} where the precise Tukey and Rudin-Keilser structures below
 ultrafilters forced by $\mathcal{P}(\om^k)/\mathrm{Fin}^{\otimes k}$
 were found.

The notion of a basic map is a strengthening of continuity, and is the same as continuity when the domain
 is a compact subset of $2^{\om}$ (see Definition \ref{defn.basic} below).
The Extension Lemma
\ref{thm.PropACtsMaps} shows that all basic Tukey reductions on some cofinal subset of an ultrafilter extend to a basic map on $\mathcal{P}(\om)$.
This will be employed   in the proof of the main theorem of this section, Theorem \ref{thm.PropB}, which shows  that, under mild assumptions, the property of having basic Tukey reductions is inherited under Tukey reducibility.
  Theorem \ref{thm.main} then follows:
Every ultrafilter Tukey reducible to a p-point has  basic, and hence continuous, Tukey reductions.
Combining  Theorem \ref{thm.main} with Theorem 10 of Raghavan in \cite{Raghavan/Todorcevic12}
yields that whenever $\mathcal{W}$ is Tukey reducible to a p-point and $\mathcal{V}$ is a q-point, then $\mathcal{W}\ge_T\mathcal{V}$ implies $\mathcal{W}\ge_{RB}\mathcal{V}$
(see Theorem
\ref{thm.dr}).

We begin with some basic definitions.
The following standard notation is used:
$2^{<\om}$  denotes the collection of finite sequences  $s:n\ra 2$, for $n<\om$.
We use $s,t,u,\dots$ to denote members of $2^{<\om}$.
Note that given  $s,t\in 2^{<\om}$,
the inclusion $s\sse t$ simply means that $s$ is an initial segment of $t$; that is,
$\dom(s)\le\dom(t)$ and
$t\re\dom(s)=s$.
Write $s\subset t$ if and only if $s\sse t$ and $s\ne t$; that is, $s$ is a proper initial segment of $t$.
We  use $a\sqsubseteq X$ for sets $a,X\sse\om$ to denote that, given their strictly increasing enumerations, $a$  is an initial segment of $X$.
$a\sqsubset X$ denotes that $a$ is a proper initial segment of $X$.

We would like to identify subsets of $\om$ with their characteristic functions, as this will greatly simplify notation in the proofs.
For
 $X\sse \om$,
we let $\chi_X$ denote the characteristic function of $X$ with domain $\om$. 
If $X$ is infinite, we shall often abuse notation and use $X$ to denote both the set $X$ and its characteristic function $\chi_X$.
Given $m<\om$,
we  let $\chi_X\re m$ denote the characteristic function of $X\cap m$ with domain $m$.
Whenever no ambiguity arises, we shall
 abuse notation and use $X\re m$ to denote both the characteristic function $\chi_X\re m$ and the set $X\cap m$.
For precision throughout, for $s\in 2^{<\om}$,
we use the notation $[[s]]$ to denote $s^{-1}(\{1\})$, the set of $i$ in the domain of $s$ for which $s(i)=1$.
For pointwise images of a function $f$ on a set $S$, we shall use the notation $f[S]$ to denote $\{f(x):x\in S\}$.
In Sections \ref{sec.Fubit} and following,
for a tree $T\sse [\om]^{<\om}$,
 we shall use $[T]$ to denote the set of all {\em branches} through $T$, meaning the collection of all
 maximal branches (if there are any finite maximal branches) and all cofinal branches (if there are any infinite branches)  through $T$.
This will not cause any ambiguity.

\begin{defin}\label{def.levinitpres}
Given  a subset $C$ of $2^{<\om}$,
we shall call a map $\hat{f}:C\ra 2^{<\om}$ \em level preserving \rm if
there is a strictly increasing sequence  $(k_m)_{m<\om}$ such that
$\hat{f}$ takes each member of
 $C\cap 2^{k_m}$ to a member of  $2^m$.
A level preserving map $\hat{f}$ is \em end-extension preserving \rm if whenever $m<n$,  $s\in C\cap 2^{k_m}$, and  $t\in C\cap 2^{k_{n}}$, then  $s\subseteq t$ implies $\hat{f}(s)\subseteq\hat{f}(t)$.
A level and end-extension preserving map shall be called {\em basic}.
 A level-preserving map
$\hat{f}$ is
{\em monotone} if for each $m<\om$ and  $s,t\in C\cap 2^{k_m}$,
$[[s]]\sse [[t]]$ implies $[[\hat{f}(s)]]\sse [[\hat{f}(t)]]$.
\end{defin}

\begin{defin}\label{defn.basic}
Let $f$ be a map from a subset $\mathcal{C}\sse 2^{\om}$ into $2^{\om}$.
We say that  $f$ is  {\em  represented by
a basic map}
if there is a strictly
increasing sequence $(k_m)_{m<\om}$ such that, letting
\begin{equation}
C=\{X\re k_m: X\in\mathcal{C},\ m<\om\},
\end{equation}
there is a basic map $\hat{f}:C\ra 2^{<\om}$
such that
for each $X\in\mathcal{C}$,
\begin{equation}
f(X)=\bigcup_{m<\om} \hat{f}(X\re k_m).
\end{equation}
In this case, we say that $\hat{f}$ \em generates \rm $f$.
If  each monotone cofinal function from an ultrafilter $\mathcal{U}$ to another ultrafilter is represented by a basic map on some cofinal subset of $\mathcal{U}$,
then we say that $\mathcal{U}$ {\em has basic Tukey reductions}.
\end{defin}

Note  that if $f\re\mathcal{C}$ is generated by a basic  map
$\hat{f}$,  then
for each $X\in\mathcal{C}$ and $m<\om$,
$f(X)\re m= \hat{f}(X\re k_m)$.
This fact will be used throughout the paper.

For a set $C\sse 2^{<\om}$, $[C]$ denotes the set of all branches through $C$.
If $C=\{X\re k_m:X\in\mathcal{C},\ m<\om\}$,
where $\mathcal{C}\sse 2^{\om}$ and $(k_m)_{m<\om}$ is a strictly increasing sequence,
then $[C]\sse 2^{\om}$.

\begin{lem}\label{lem.basic1}
Let  $\hat{f}$ be a monotone basic map with domain  $C\sse 2^{<\om}$.
Then $\hat{f}$
induces a monotone
 continuous map $f^*$ on $[C]$ by $f^*(X)=\bigcup_{m<\om}\hat{f}(X\re k_m)$, for $X\in[C]$.
Further, if  $\hat{f}$ generates $f$ on $\mathcal{C}$ and $C\contains\{X\re k_m:X\in\mathcal{C},\ m<\om\}$, then $f^*\re\mathcal{C}=f\re\mathcal{C}$; hence $f\re \mathcal{C}$ is continuous.
\end{lem}

\begin{proof}
That $f^*$ is continuous on $[C]$ is trivial, since $\hat{f}$ is  basic.
Since $\hat{f}$ is monotone, it follows that $f^*$ is monotone.
If $\hat{f}$ generates $f$ on $\mathcal{C}$, then trivially $f^*\re \mathcal{C}$ is simply $f\re\mathcal{C}$.
\end{proof}

In the next results,  members of ultrafilters are identified with their characteristic functions.
Then each function $f:\mathcal{U}\ra\mathcal{V}$, where $\mathcal{U}$ and $\mathcal{V}$ are ultrafilters  on $\om$,  is identified with  the function  $F:\{\chi_U:U\in\mathcal{U}\}\ra \{\chi_V:V\in\mathcal{V}\}$, 
where $F(\chi_U)=\chi_{f(U)}$.
Rather than using more notation, we  shall use $f$ to denote both of these functions. 
Since all ultrafilters  considered in this paper are nonprincipal, the notation $f(U)\sse f(U')$  will always refer to $f(U)$  and $f(U')$ as subsets of $\om$.

\begin{lem}[Extension]\label{thm.PropACtsMaps}
Suppose $\mathcal{U}$ and $\mathcal{V}$ are nonprincipal ultrafilters,
$f:\mathcal{U}\ra\mathcal{V}$
is a  monotone cofinal map, and
there is a cofinal subset $\mathcal{C}\sse\mathcal{U}$ such that
  $f\re\mathcal{C}$ is represented by a  monotone basic map $\hat{f}$.
   Let  $(k_m)_{m<\om}$ be the strictly increasing sequence such that the domain, $C$, of $\hat{f}$
   is contained in $\bigcup_{m<\om}2^{k_m}$.
Then there is a  continuous  monotone map $\tilde{f}:2^{\om}\ra 2^{\om}$ such that
\begin{enumerate}
\item
$\tilde{f}$ is represented by a monotone basic map $\hat{g}$ on $\bigcup_{m<\om}2^{k_m}$;
\item
 $\tilde{f}\re \mathcal{C}=f\re\mathcal{C}$;  and
\item
$\tilde{f}\re\mathcal{U}$ is a cofinal map from $\mathcal{U}$ to $\mathcal{V}$.
\end{enumerate}
\end{lem}

\begin{proof}
Let $\hat{f}$ be a monotone basic  map
generating $f\re\mathcal{C}$,
and let
  $(k_m)_{m<\om}$ be the levels on which $\hat{f}$ is defined.
Thus, the domain of $\hat{f}$ is
$C=\{X\re k_m:X\in\mathcal{C},\ m<\om\}$,
and
 for each $s\in C\cap 2^{k_m}$,
 $\hat{f}(s)\in 2^m$.

\begin{claim}\label{claim.A2}
There is a monotone basic map $\hat{g}$, with domain $\bigcup_{m<\om}2^{k_m}$, which  generates a function $\tilde{f}:2^{\om}\ra 2^{\om}$ such that  $\tilde{f}\re\mathcal{C}=f\re\mathcal{C}$.
\end{claim}

\begin{proof}
Since $\mathcal{C}$ is cofinal in $\mathcal{U}$ and $\mathcal{U}$ is nonprincipal,
 the
 finite sequence of zeros of length $k_m$
is in $C$, for each $m<\om$.
Let $D=\bigcup_{m<\om}2^{k_m}$ and define $\hat{g}$ on $D$ as follows:
For $t\in 2^{k_m}$,
define
$\hat{g}(t)$  to be the function from $m$ into $2$ such that
for $i\in m$,
\begin{equation}
\hat{g}(t)(i) = \max\{\hat{f}(s)(i):
 s\in C,\ |s|\le k_m, \mathrm{\ and\ }[[s]]\sse [[t]]\}.
\end{equation}
That is, $\hat{g}(t)(i)=1$ if and only if there is some $s\in C$ such that  $|s|\le k_m$, $[[s]]\sse [[t]]$, and $\hat{f}(s)(i)=1$.
It follows from the  definition that $\hat{g}$ is monotone and  level preserving.
Since $\hat{f}$ is monotone, $\hat{g}\re C$ equals $\hat{f}$.

To see that $\hat{g}$ is end-extension preserving, suppose $t\subset t'$, where $t\in 2^{k_m}$ and $t'\in 2^{k_{m'}}$ for some $m<m'$.
Fix $i<m$.
Suppose that $\hat{g}(t')(i)=1$.
Then there is some
 $s'\in C\cap  2^{k_n}$ such that $i< n\le m'$, $[[s']]\sse [[t']]$, and $\hat{f}(s')(i)=1$.
Letting  $j=\min\{m,n\}$
and $s=s'\re k_j$, we see that
 $s\in C$ and $[[s]]\sse [[t'\re k_m]]$, where $t'\re k_m=t$; moreover, $\hat{f}(s)(i)=1$, since
$\hat{f}(s)=\hat{f}(s')\re j$.
It follows that $\hat{g}(t)(i)=1$.
On the other hand, if $\hat{g}(t')(i)=0$,
then by the definition of $\hat{g}$,
$\hat{f}(s)(i)=0$ for all $s\in C\cap\bigcup_{n\le m}2^{k_n}$ such that $[[s]]\sse [[t]]$.
Hence, $\hat{g}(t)(i)=0$.
Therefore, $\hat{g}(t')\re m=\hat{g}(t)$.

Now define $\tilde{f}:2^{\om}\ra 2^{\om}$ by
\begin{equation}
\tilde{f}(Z)=\bigcup_{m<\om}\hat{g}(Z\re k_m).
\end{equation}
Then  $\tilde{f}$ is generated by the basic  map $\hat{g}$.
It follows that $\tilde{f}$ is monotone.
Since  $\hat{g}\re C$  equals $\hat{f}$,
it follows that $\tilde{f}\re\mathcal{C}=f\re\mathcal{C}$.
\end{proof}

Thus, $\tilde{f}$ is continuous on $2^{\om}$
 and (1) and (2) of the Lemma hold.
To show (3), it suffices to show that $\tilde{f}\re\mathcal{U}$ has range inside of $\mathcal{V}$, since $\tilde{f}\re\mathcal{C}$ equals $f\re\mathcal{C}$ which is monotone and cofinal in $\mathcal{V}$.
Let $U\in\mathcal{U}$ be given.
Fix $S\in\mathcal{C}$ such that $U\contains S$.
Then $\tilde{f}(U)\contains \tilde{f}(S)=f(S)\in\mathcal{V}$, which concludes the proof.
\end{proof}

In the main theorem of this section,
we show that the property of having basic Tukey reductions  is inherited under Tukey reducibility below any ultrafilter having monotone basic Tukey reductions and satisfying the property $(*)$ below.

\begin{thm}\label{thm.PropB}
Suppose that the ultrafilter
$\mathcal{U}$ has monotone
basic Tukey reductions.
Suppose  further that for   each monotone cofinal map $f$ from $\mathcal{U}$ to another ultrafilter, there is some
 cofinal subset $\mathcal{C}\sse\mathcal{U}$
 such that $f\re\mathcal{C}$ is represented by a  monotone basic function  on some  levels $(k_m)_{m<\om}$ satisfying the following property:
\begin{enumerate}
\item[$(*)$]
For each $X\in\overline{\mathcal{C}}$ and each $m<\om$, there is a $Z\in\mathcal{C}$ such that
$Z\contains X$ and $Z\re k_m=X\re k_m$.
\end{enumerate}
Then every ultrafilter $\mathcal{V}$ Tukey reducible to $\mathcal{U}$ also has
 basic Tukey reductions.
\end{thm}

\begin{proof}
Suppose that  $\mathcal{U}$ satisfies the hypotheses and let $\mathcal{V}\le_T\mathcal{U}$.
We  may assume that $\mathcal{U}$ is nonprincipal, for the result holds immediately  if $\mathcal{U}$ is principal.
By  Lemma \ref{thm.PropACtsMaps}, there is a
 map $\tilde{f}:2^{\om}\ra 2^{\om}$  generated by a monotone basic  map $\hat{f}:\bigcup_{m<\om}2^{k_m}\ra 2^{<\om}$, for some  increasing sequence $(k_m)_{m<\om}$, such that
$\tilde{f}\re \mathcal{U}:\mathcal{U}\ra\mathcal{V}$ is a cofinal map.
We shall let
  $f$ denote the restricted map  $\tilde{f}\re\mathcal{U}$.
Suppose $\mathcal{W}\le_T\mathcal{V}$, and let
$h:\mathcal{V}\ra\mathcal{W}$ be a monotone cofinal map.
Extend $h$ to the map $\tilde{h}: 2^{\om} \ra 2^{\om}$  defined as follows:
For each $X\in 2^{\om}$, let
\begin{equation}
\tilde{h}(X)=\bigcap\{h(V):V\in\mathcal{V}\mathrm{\ and\ }V\contains X\}.
\end{equation}
Notice that $\tilde{h}$ is monotone.
Furthermore, it follows from $h$ being monotone that $\tilde{h}\re \mathcal{V}=h$.

Define $\tilde{g}=\tilde{h}\circ \tilde{f}$.
Then $\tilde{g}:2^{\om}\ra 2^{\om}$ and is monotone.
Letting $g$ denote $\tilde{g}\re \mathcal{U}$,
we see that
$g=h\circ f$; hence
 $g:\mathcal{U}\ra\mathcal{W}$ is a monotone cofinal map.
By the hypotheses,
there is a cofinal subset $\mathcal{C}\sse\mathcal{U}$ and
a monotone basic map  $\hat{g}:C\ra 2^{<\om}$  generating  $g\re\mathcal{C}$ such that  $(*)$ holds,
where $(k_m)_{m<\om}$ is the strictly increasing sequence associated with $\hat{g}$ and
$C=\{X\re k_m:X\in\mathcal{C}\mathrm{\ and\ } m<\om\}$.

Without loss of generality, we may assume that $\hat{f}$ and $\hat{g}$ are defined on the same  levels $(k_m)_{m<\om}$ and that $k_0=0$:
For if $\hat{g}$ is defined on $\{X\re j_m:X\in\mathcal{C}$ and $m<\om\}$,
we can take $l_m=\max(k_m,j_m)$ and
 define $\hat{f}'(s)=\hat{f}(s\re k_m)$ for $s\in 2^{l_m}$
and $\hat{g}'(X\re l_m)=\hat{g}(X\re j_m)$  for $X\in\mathcal{C}$ and $m<\om$.
Notice that whenever
$s\in C\cap 2^{k_m}$ and  $s \sqsubset X\in\mathcal{C}$, then
 $\hat{f}(s)=f(X)\re m$
and $\hat{g}(s)=g(X)\re m$.

Define \begin{equation}
D=\{\hat{f}(s):s\in C\}
\textrm{\ and  \ } \mathcal{D}=f[\mathcal{C}].
\end{equation}
Notice that in fact $D=\{Y\re m:Y\in\mathcal{D},\ m<\om\}$, and
  $\mathcal{D}$ is cofinal in $\mathcal{V}$ since $f:\mathcal{U}\ra\mathcal{V}$ is monotone cofinal and $\mathcal{C}$ is a cofinal subset of $\mathcal{U}$.
Let $\overline{\mathcal{C}}$ denote the closure of $\mathcal{C}$ in the topological space $2^{\om}$.
Since $\tilde{f}$ is continuous on the compact space $2^{\om}$ and $f\re\mathcal{C}=\tilde{f}\re \mathcal{C}$, it follows that
$\overline{\mathcal{D}}=\overline{f[\mathcal{C}]}=
\tilde{f}[\overline{\mathcal{C}}]$.

\begin{claim1}\label{claim.Bii}
For each $Y\in\overline{\mathcal{D}}$ and  each $m<\om$, there is an $\tilde{m}\ge m$ satisfying the following:
For each $Z\in\overline{\mathcal{C}}$ such that $\tilde{f}(Z)\re \tilde{m}=Y\re \tilde{m}$,
 there is an $X\in\overline{\mathcal{C}}$ such that
$\tilde{f}(X)=Y$ and
 $X \re k_m=Z \re k_m$.
\end{claim1}

\begin{proof}
Let $Y\in\overline{\mathcal{D}}$ and suppose the claim fails.
Then
there is an $m$ such that for each $n\ge m$,
there is a $Z_n\in\overline{\mathcal{C}}$ such that $\tilde{f}(Z_n)\re n=Y\re n$,
but for each $X\in\overline{\mathcal{C}}$ such that $\tilde{f}(X)=Y$,
$Z_n \re k_m\ne X\re k_m$.
$\overline{\mathcal{C}}$ is compact,
so there is a subsequence $(Z_{n_i})_{i<\om}$ which converges to some $X\in\overline{\mathcal{C}}$.
Since $\tilde{f}$ is continuous, $\tilde{f}(Z_{n_i})$ converges to $\tilde{f}(X)$.
Since   for each $i<\om$, $\tilde{f}(Z_{n_i})\re n_i=Y\re n_i$,
it follows that $\tilde{f}(Z_{n_i})$ converges to $Y$.
Therefore, $\tilde{f}(X)=Y$.
Further, since $Z_{n_i}\ra X$, there is a $j$ such that for all $i\ge j$, $Z_{n_i}\re k_m=X\re k_m$.
But this is a contradiction since $X\in\overline{\mathcal{C}}$ and
$\tilde{f}(X)=Y$.
\end{proof}

\begin{claim2}
There is a strictly increasing sequence
$(j_m)_{m<\om}$  such that
for each  $m<\om$,  for all $Y\in\overline{\mathcal{D}}$
and
 $Z\in\overline{\mathcal{C}}$ with $\tilde{f}(Z)\re j_{m}=Y\re j_{m}$,
there is an $X\in\overline{\mathcal{C}}$ such that
$\tilde{f}(X)=Y$ and
 $X\re k_m=Z\re k_m$.
\end{claim2}

\begin{proof}
Let $j_0=0$ and note that
 $j_0$ vacuously satisfies the claim.
Now suppose that $m\ge 1$ and suppose we have chosen $j_0<\dots<j_{m-1}$ satisfying the claim.
For each   $Y\in\overline{\mathcal{D}}$, there is an $\tilde{m}(Y)\ge m$ satisfying Claim 1.
The finite characteristic functions $Y\re \tilde{m}(Y)$ determine basic open sets in $2^{\om}$, and the union of these open sets (over all $Y\in\overline{\mathcal{D}}$)
 covers $\overline{\mathcal{D}}$.
Since $\overline{\mathcal{D}}$ is compact,
 there is a finite subcover, determined by
some
$Y_0\re \tilde{m}(Y_0),\dots, Y_l\re \tilde{m}(Y_l)$.
Take $j_{m}>\max\{j_{m-1},\tilde{m}(Y_0),\dots,
\tilde{m}(Y_l)\}$.
By this inductive construction, we obtain a sequence  $(j_m)_{m<\om}$ which satisfies the claim.
\end{proof}

Recall  that the notation $\tilde{g}(X)\re m$ (and analogues) 
has the dual meaning of both a characteristic function with domain $m$ and also 
 the set $[[\tilde{g}(X)]]\cap m$.
 For clarity, 
 in the next few claims,  we shall use the precise notation of $[[\tilde{g}(X)]]\cap m$ to denote the set, and $\tilde{g}(X)\re m$ to denote the characteristic function.

\begin{claim3}
For each $X\in \overline{\mathcal{C}}$ and $m<\om$,
\begin{equation}\label{eq.starintersection}
[[\tilde{g}(X)]]\cap m\sse [[\hat{g}(X\re k_m)]].
\end{equation}
\end{claim3}

\begin{proof}
Let  $X\in \overline{\mathcal{C}}$ and $m<\om$ be given,
and let  $s$ denote $X\re k_m$.
Then $s\in C$, and
for any $Z\in\mathcal{C}$ such that $Z\re k_m=s$,
we have that
$\tilde{g}(Z)\re m=\hat{g}(s)$.
Since  the property $(*)$ on $\mathcal{C}$ implies
there is  a $Z\in\mathcal{C}$ such that $Z\contains X$ and $Z\re k_m=s$,
and
since $\tilde{g}$ is monotone, it follows that
$$
[[\tilde{g}(X)]]\cap m\sse
\bigcap\{[[\tilde{g}(Z)]]\cap  m:Z\in\mathcal{C}\mathrm{\  and\ }Z\contains X\}
=
[[\hat{g}(s)]].
$$
\end{proof}


\begin{claim4}\label{claim.FactConverge}
Let $Y\in\overline{\mathcal{D}}$ and $m$ be given,
 and let $t=Y\re j_m$.
Then
$$
[[\tilde{h}(Y)]]\cap m\sse [[\hat{g}(s)\re m]],
$$
for each $s\in C\cap 2^{k_{j_m}}$ such that $\hat{f}(s)=t$.
\end{claim4}

\begin{proof}
Let $Y\in\overline{\mathcal{D}}$ and $m$ be given,
and let $t=Y\re j_m$.
Let $s$ be any member of $C\cap 2^{k_{j_m}}$ such that $\hat{f}(s)=t$.
By Claim 2, there is an $X\in\overline{\mathcal{C}}$ such that
$\tilde{f}(X)=Y$ and
$X \re k_m= s\re k_m$.
To prove the claim, we shall show that the following holds:
\begin{equation}\label{eq.ghi}
[[\tilde{h}(Y)]]\cap  m
=
[[\tilde{g}(X)]]\cap m
\sse
[[\hat{g}(X\re k_m)]]
=
[[\hat{g}(s)\re m]].
\end{equation}
The first equality  holds since
$\tilde{f}(X)=Y$ and
$\tilde{h}\circ\tilde{f}(X)=\tilde{g}(X)$.
The last  equality holds since
$\hat{g}(X \re k_m)=\hat{g}(s\re k_m)=\hat{g}(s)\re m$.
The inclusion holds by Claim 3.
\end{proof}

Finally, we define the finitary function $\hat{h}$ which will represent $h$ on $\mathcal{D}$.
Let $D'=\{t\in D: \exists m<\om\, (|t|=j_m)\}$.
For $t\in D'\cap 2^{j_m}$,  define $\hat{h}(t)$ to be the function from $m$  into $2$  such that for $i\in m$,
\begin{equation}
\hat{h}(t)(i)=\min\{\hat{g}(s)(i): s\in C\cap 2^{k_{j_m}}\mathrm{\ and\ }\hat{f}(s)=t\}.
\end{equation}
In words, $\hat{h}(t)$ is the characteristic function  with domain $m$ of the intersection of the  subsets  $a$ of $m$  for which there is some
$s\in C\cap 2^{k_{j_m}}$ with $\hat{f}(s)=t$
such that
$\hat{g}(s)\re m$   is the characteristic function of $a$ with domain $m$.
By definition, $\hat{h}$ is level preserving.

\begin{claim5}
 $\hat{h}$ is basic and generates $h\re \mathcal{D}$.
\end{claim5}

\begin{proof}
Let $Y\in\mathcal{D}$, $Z$ be a member of $\mathcal{C}$ such that $\tilde{f}(Z)=Y$,
and $m<\om$ be given.
Since $Z\in\mathcal{C}$, $\tilde{f}(Z)=f(Z)$.
Let $t=Y\re j_m$ and $u=Z\re k_{j_m}$.
Then $\hat{f}(u)=t$,
so $[[\hat{g}(u)\re m]]\contains [[\hat{h}(t)]]$.
Since $Z\in\mathcal{C}$, $[[g(Z)]]\cap m=[[\hat{g}(u)\re m]]$.
Thus,
\begin{equation}
[[h(Y)]]\cap m
= [[h\circ f(Z)]]\cap m
=[[g(Z)]]\cap  m
=[[\hat{g}(u)\re m ]]\contains [[\hat{h}(t)]].
\end{equation}

Now suppose $s\in C\cap 2^{k_{j_m}}$  and $\hat{f}(s)=t$.
By Claim 4, since $\tilde{h}\re\mathcal{V}=h$,
we see that $[[h(Y)]]\cap m=[[\tilde{h}(Y)]]\cap  m\sse [[\hat{g}(s)\re m]]$.
Since $s$ was arbitrary, it follows that $[[h(Y)]]\cap m\sse [[\hat{h}(t)]]$.
Therefore, $\hat{h}(Y\re j_m)=h(Y)\re m$.

Thus, $\hat{h}$ generates $h$ on $\mathcal{D}$.
It  follows that $\hat{h}$ is end-extension preserving:
If $t\subset t'$ are members of $D'$ of lengths $j_m$ and $j_{m'}$, respectively,
then letting $Y$ be any member of $\mathcal{D}$ such that $t'$ is an initial segment of $Y$,
we see that
$\hat{h}(t)=h(Y)\re m= (h(Y)\re m')\re m
=\hat{h}(t')\re m$.
Therefore, $\hat{h}$ is basic.
\end{proof}

Thus, $h\re\mathcal{D}$ is  generated by the basic map $\hat{h}$ on $D'$.
Thus, $\mathcal{V}$ has basic Tukey reductions.
\end{proof}

Every p-point has monotone basic Tukey reductions satisfying the additional property $(*)$ of Theorem \ref{thm.PropB},
as was shown in the proof of Theorem 20 of \cite{Dobrinen/Todorcevic11}, the cofinal set $\mathcal{C}$ there being  of the simple form $\mathcal{P}(X)\cap\mathcal{U}$ for some  $X\in\mathcal{U}$.
Hence, the following theorem holds.

\begin{thm}\label{thm.main}
Every ultrafilter Tukey reducible to a  p-point
 has basic, and hence continuous, Tukey reductions.
\end{thm}

Recall that an ultrafilter $\mathcal{V}$ is Rudin-Blass reducible to an ultrafilter $\mathcal{W}$ if there is a finite-to-one map $h:\om\ra\om$ such that $\mathcal{V}=h(\mathcal{W})$.
Thus, Rudin-Blass reducibility implies Rudin-Keisler reducibility.
Our Theorem \ref{thm.main} combines with
Theorem 10 of Raghavan in \cite{Raghavan/Todorcevic12} (see Theorem \ref{thm.RT} for the statement) to yield the following.

\begin{thm}\label{thm.dr}
Suppose $\mathcal{U}$ is Tukey reducible to a p-point.
Then  for each q-point $\mathcal{V}$,
$\mathcal{V}\le_T\mathcal{U}$ implies $\mathcal{V}\le_{RB}\mathcal{U}$.
\end{thm}

\begin{rem}
Stable ordered-union ultrafilters are the analogues of p-points on the base set $\FIN=[\om]^{<\om}\setminus\{\emptyset\}$ (see \cite{Blass87}).
In
 Theorems 71 and 72 of \cite{Dobrinen/Todorcevic11},
 it was shown that for each stable ordered union ultrafilter $\mathcal{U}$, both $\mathcal{U}$ and its projection $\mathcal{U}_{\min,\max}$ have  continuous Tukey reductions, with respect to the metric topology on the Milliken space of infinite block sequences.
It is of interest that  the ultrafilter $\mathcal{U}_{\min,\max}$ is rapid, but is neither a p-point nor a q-point, and condition $(*)$ of Theorem \ref{thm.PropB} is satisfied. 
Further,  all ultrafilters selective for some topological Ramsey space have monotone basic Tukey reductions, under a mild assumption which is satisfied in all known topological Ramsey spaces, 
 Dobrinen and Trujillo showed in Theorem  56 of
 in  \cite{Dobrinen/Mijares/Trujillo14}.
Many such ultrafilters are not p-points.

It should be the case that by arguments similar to those in  Theorem \ref{thm.PropB},
one can prove that every ultrafilter Tukey reducible to some stable ordered union ultrafilter, or more generally, any ultrafilter selective for some topological Ramsey space, also has continuous Tukey reductions.
We leave this as an open problem in Section \ref{sec.op}.
\end{rem}


\section{Iterated Fubini products of ultrafilters represented as ultrafilters generated by $\vec{\mathcal{U}}$-trees on  fronts}\label{sec.Fubit}

Fubini products of ultrafilters on base set $\om$ are commonly viewed as ultrafilters on base set $\om\times\om$.
As was pointed out to us by Todorcevic, Fubini products of nonprincipal ultrafilters on base set $\om$ may also be viewed as ultrafilters on base set $[\om]^2$.
This view leads well to precise investigations of ultrafilters constructed
 by iterating the Fubini product construction.
In this section, we review Fubini products of ultrafilters and countable iterations of this construction.
After reviewing the notion of front,  we then
show how every ultrafilter obtained by iterating the Fubini product construction can be viewed as an ultrafilter generated by certain subtrees of a base set which is a tree, and in particular,  a  front.
This section is a primer for the work in Section \ref{sec.BasicCofinalFubit}.

\begin{defin}\label{defin.cross}
Let $\mathcal{U}$ and $\mathcal{V}_n$ ($n<\om$) be ultrafilters on $\om$.
The \em Fubini product \rm of $\mathcal{V}_n$ over $\mathcal{U}$, denoted
$\lim_{n\ra\mathcal{U}}\mathcal{V}_n$, is defined as follows:
\begin{equation}
\lim_{n\ra\mathcal{U}}\mathcal{V}_n= \{A\sse\om\times\om:\{n\in\om:\{j\in\om:(n,j)\in A\}\in\mathcal{V}_n\}\in\mathcal{U}\}.
\end{equation}
When all $\mathcal{V}_n=\mathcal{V}$,
then we let $\mathcal{U}\cdot\mathcal{V}$ denote
$\lim_{n\ra\mathcal{U}}\mathcal{V}_n$.
\end{defin}

Let $\mathcal{U}$ and $\mathcal{V}$ be ultrafilters on countable base sets $I$ and $J$, respectively.
We say that $\mathcal{U}$ is {\em isomorphic} to $\mathcal{V}$ if there exists a bijection $\pi:I\ra J$
such that $\mathcal{V}=\{\pi[U]:U\in\mathcal{U}\}$.
Up to isomorphism, Definition \ref{defin.cross} also defines Fubini products of ultrafilters on arbitrary countable infinite sets.

The Fubini product construction  can be iterated countably many times, each time producing an ultrafilter.
For example,
given an ultrafilter $\mathcal{V}$, let $\mathcal{V}^1$ denote $\mathcal{V}$, and
let $\mathcal{V}^{n+1}$ denote $\mathcal{V}\cdot\mathcal{V}^n$.
Naturally, $\mathcal{V}^{\om}$ denotes $\lim_{n\ra\mathcal{V}}\mathcal{V}^n$.
Continuing in this manner, we obtain $\mathcal{V}^{\al}$, for all $2\le \al<\om\cdot 2$.
At this point, it is ambiguous what is meant by $\mathcal{V}^{\om\cdot 2}$.
It is standard practice for  countable a limit ordinal $\al$ to let $\mathcal{V}^{\al}$ denote any ultrafilter  constructed by choosing (arbitrarily) an increasing  sequence  $(\al_n)_{<\om}$ converging to $\al$ and defining $\mathcal{V}^{\al}$ to be $\lim_{n\ra\mathcal{V}}\mathcal{V}^{\al_n}$, but  this is ambiguous, since the choice of the sequence $(\al_{n})_{n<\om}$ is completely arbitrary.

However,
each countable iteration of Fubini products of ultrafilters (including the choice of sequence at limit stages) can be represented as an
ultrafilter generated by $\vec{\mathcal{U}}$-trees (see Definition \ref{def.Utree})
on
 a base set which is
a front.
This representation is unambiguous at limit stages.
For this reason,
Theorem  \ref{thm.allFubProd_p-point_cts}
in the next section,
showing that iterations of Fubini products of p-points have  Tukey reductions which are as close to continuous as possible,
will be carried out in the setting of $\vec{\mathcal{U}}$-trees.

We now  recall the definition of  front.
The reader desiring more background on fronts and $\vec{\mathcal{U}}$-trees than  presented here is referred to \cite{TodorcevicBK10}, pages 12 and 190, respectively.

\begin{defin}\label{def.front}
A family $B$ of finite subsets of some infinite subset $I$ of $\om$ is called a \em front \rm on $I$ if
\begin{enumerate}
\item
$a\not\sqsubset b$ whenever $a, b$ are in $B$; and
\item
For every infinite $X\sse I$ there exists $b\in B$ such that $b\sqsubset X$.
\end{enumerate}
\end{defin}

Recall the following standard set-theoretic notation:  $[\om]^k$  denotes the collection of $k$-element subsets of $\om$, $[\om]^{<k}$ denotes the collection of subsets of $\om$ of size less than $k$,
and $[\om]^{\le k}=[\om]^{<k+1}$.
It is easy to see that for each $k<\om$, $[\om]^{k}$ is a front.

Every front  is lexicographically well-ordered, and hence has a unique lexicographic rank associated with it, namely the ordinal length of its lexicographical well-ordering.
For example, $\rank(\{\emptyset\})=1$,  $\rank([\om]^1)=\om$, and $\rank([\om]^2)=\om\cdot\om$.
We shall usually drop the adjective `lexicographic' when talking about ranks of fronts.

Given a front $B$,
for each $n\in \om$,  we define $B_n=\{b\in B:n=\min(b)\}$
and  $B_{\{n\}}=\{b\setminus \{n\}:b\in B_n\}$.
Then $B=\bigcup_{n\in \om} B_n$, and each $B_n=\{\{n\}\cup a:a\in B_{\{n\}}\}$.
Note that for each $n\in \om$,
$B_{\{n\}}$ is a front on $[n+1,\om)$ with rank strictly less than the rank of $B$.
Conversely, given any collection of fronts $B_{\{n\}}$ on $[n+1,\om)$,
the union $\bigcup_{n\in \om} B_n$ is a front on $\om$, where $B_n$ is defined as above to be $\{\{n\}\cup a:a\in B_{\{n\}}\}$.

Given any  front $B$, we
let $\hat{B}$ denote the collection of all initial segments of members of $B$.
Let $\hat{B}^-$ denote the collection of all proper initial segments of members of $B$;
that is, $\hat{B}^-=\hat{B}\setminus B$.
Both $\hat{B}$ and $\hat{B}^-$ form  trees under the partial ordering $\sqsubseteq$.

\begin{defin}\label{def.Utree}
Given a front $B$ and  a sequence $\vec{\mathcal{U}}=(\mathcal{U}_c:c\in \hat{B}^-)$ of nonprincipal ultrafilters $\mathcal{U}_c$ on $\om$,
a {\em $\vec{\mathcal{U}}$-tree} is a  tree $T\sse \hat{B}$ such that $\emptyset\in T$ and  for each $c\in T\cap \hat{B}^-$,
 $\{n\in\om:c\cup \{n\}\in T\}\in\mathcal{U}_c$.
\end{defin}

\begin{notn}\label{notn.Utrees}
Given a  front $B$ and a sequence $\vec{\mathcal{U}}=(\mathcal{U}_c:c\in \hat{B}^-)$ of nonprincipal ultrafilters on $\om$,
let $\mathfrak{T}=\mathfrak{T}(\vec{\mathcal{U}})$ denote the collection of all $\vec{\mathcal{U}}$-trees.
For any $c\in \hat{B}^-$ and $T\in\mathfrak{T}$,
let
$T_c=\{t\in T:t\sqsubseteq c$ or $t\sqsupset c\}$,
 the tree with stem $c$ consisting of all nodes in $T$ comparable with $c$.
If $T$ is a  $\mathcal{U}$-tree, then the set of maximal branches through $T$, denoted $[T]$, is exactly $T\cap B$.
\end{notn}

Todorcevic pointed out to us the following correspondence between Fubini iterates and ultrafilters on fronts.
Start by fixing a collection
 $\mathcal{P}_0$ of nonprincipal ultrafilters on $\om$.
Given $\al<\om_1$, define $\mathcal{P}_{\al+1}=\{\lim_{n\ra\mathcal{U}}\mathcal{V}_n:\mathcal{U}\in\mathcal{P}_0$ and $\mathcal{V}_n\in\mathcal{P}_{\al}\}$.
For each limit ordinal $\al$, define $\mathcal{P}_{\al}=\bigcup_{\beta<\al}\mathcal{P}_{\beta}$.
Then $\mathcal{P}_{<\om_1}:=\bigcup\{\mathcal{P}_{\al}:\al<\om_1\}$ is the collection of all iterated Fubini products of nonprincipal ultrafilters on $\om$.
Each $\mathcal{W}\in\mathcal{P}_{<\om_1}$ has a well-defined notion of rank, namely rank$(\mathcal{W})$ is the least $\al<\om_1$ for which it is a member of $\mathcal{P}_{\al}$.
The following lemma will be used in the next section, with $\mathcal{P}_0$ being the collection of all p-points.

\begin{lem}\label{lem.precise.connection}
If $\mathcal{W}\in\mathcal{P}_{<\om_1}$,
 then there is a  front $B$ and ultrafilters $\mathcal{U}_c\in\mathcal{P}_0$, $c\in \hat{B}^-$, such that $\mathcal{W}$ is isomorphic to the ultrafilter on $B$ generated by the $(\mathcal{U}_c:c\in \hat{B}^-)$-trees.
\end{lem}

\begin{proof}
We prove by induction on $\al<\om_1$ that the fact holds for every ultrafilter in $\mathcal{P}_{\al}$.
If $\mathcal{W}\in\mathcal{P}_0$, then $\mathcal{W}$ is represented on the  front $B=[\om]^1$ via the obvious isomorphism $n\mapsto \{n\}$.

Let $1\le \al<\om_1$ and assume the fact holds for each ultrafilter in $\bigcup_{\gamma<\al}\mathcal{P}_{\gamma}$.
If $\al$ is a limit ordinal, then there is nothing to prove, so assume $\al=\beta+1$ for some $\beta<\om_1$.
Suppose that $\mathcal{W}\in\mathcal{P}_{\al}$.
Then $\mathcal{W}=\lim_{n\ra\mathcal{U}}\mathcal{W}_n$, where $\mathcal{U}\in\mathcal{P}_0$
and  for each $n$, $\mathcal{W}_n\in\mathcal{P}_{\beta}$.
By the induction hypothesis, for each $n<\om$
there is a  front $B_{\{n\}}$ on $\om$ and there are ultrafilters $\mathcal{U}_{n,c}\in \mathcal{P}_0$, $c\in {\widehat{B_{\{n\}}}}{}^-$,  such that $\mathcal{W}_n$ is isomorphic to the ultrafilter generated by  $(\mathcal{U}_{n,c}:c\in {\widehat{B_{\{n\}}}}{}^-)$-trees on $B_{\{n\}}$.
Without loss of generality, we may assume that $B_{\{n\}}$ is a  front on $[n+1,\om)$.
In the standard way, we glue the fronts together to obtain a new  front:
Let  $B=\bigcup_{n<\om}\{\{n\}\cup b:b\in  B_{\{n\}}\}$.
Then $B$ is a  front on $\om$.

Let $\mathcal{U}_{\emptyset}=\mathcal{U}$ and $\mathcal{U}_{\{n\}\cup c}=\mathcal{U}_{n,c}$, for each $c\in B_{\{n\}}$.
It is straightforward to check that $\mathcal{W}$ is isomorphic to the ultrafilter on $B$ generated by the $\lgl \mathcal{U}_b:b\in \hat{B}^-\rgl$-trees.
\end{proof}

The case when $\al=2$ of the proof of Lemma \ref{lem.precise.connection} yields the following result, which seems interesting in its own right.

\begin{lem}\label{fact.fubprod}
The Fubini product $\lim_{n\ra\mathcal{U}}\mathcal{V}_n$  of nonprincipal  ultrafilters on $\om$ is isomorphic to the ultrafilter on $B=[\om]^2$ generated by  $\vec{\mathcal{U}}=(\mathcal{U}_c:c\in[\om]^{\le 1})$-trees,
where $\mathcal{U}_{\emptyset}=\mathcal{U}$ and for each $n\in\om$,
$\mathcal{U}_{\{n\}}=\mathcal{V}_n$.
\end{lem}


\section{Basic cofinal maps on iterated Fubini products of p-points}\label{sec.BasicCofinalFubit}

Fubini products of p-points do not in general have continuous Tukey reductions, as pointed out in the introduction (see below for more discussion).
However, we will show that they do have canonical cofinal maps satisfying many of the properties of continuous maps, which we call {\em basic}  (see Definition \ref{defn.basicTredtree} below).
Making use of the natural
 representation of a Fubini iterate of p-points as an ultrafilter generated by $\vec{\mathcal{U}}$-trees on some front $B$ (recall Lemma  \ref{lem.precise.connection}),
we show  in
Theorem \ref{thm.allFubProd_p-point_cts} that countable iterates of  Fubini products of p-points have  {\em basic Tukey reductions}.
Such
 Tukey reductions are continuous on the space $2^{\hat{B}}$ with the Cantor topology, where $\hat{B}$ is the tree consisting of all initial segments of members of the front $B$.
This extends a  key property of p-points (recall Theorem \ref{thm.ppoint}) to  a large class of ultrafilters.
Theorem
\ref{thm.allFubProd_p-point_cts} will be applied
in Sections \ref{sec.directapp} and \ref{sec.finitegen}.

Corollary 11 of Raghavan in \cite{Raghavan/Todorcevic12}
shows that it is impossible for a Fubini product of two non-isomorphic selective ultrafilters to have continuous cofinal Tukey reductions.
 The next proposition shows that Fubini products of nonprincipal ultrafilters  do not have monotone basic Tukey reductions given by an approximating map on the finite subsets of the base for the ultrafilter.
Thus, it is impossible to attain  
 Theorem \ref{thm.ppoint}
 for Fubini products of nonprincipal ultrafilters.

Recall that the Cantor topology on $2^{\om\times\om}$ is the topology generated by basic open sets of the form $\{h\in 2^{\om\times\om}:s\subset h\}$,
where $s$ is a function from some finite subset of $\om\times\om$ into $2$.
Equivalently, letting $\lgl w_i:i<\om\rgl$ be any linear order of the members of $\om\times\om$ in order type $\om$,
the Cantor topology on $2^{\om\times\om}$ is  generated by basic open sets of the form
$\{h\in 2^{\om\times\om}:s\subset h\}$,
where $s$ is a  function from $\lgl w_i:i<k\rgl$ into $2$ for some $k<\om$.
For $k<\om$, let $W_k$ denote $\{w_i:i<k\}$,
so that $2^{W_k}$ denotes the set of all functions with domain $\{w_i:i<k\}$ into $2$.
 Given $k<\om$ and $X\sse\om\times\om$,
let $X \re W_k$ denote the characteristic function of  $\{w_i\in X:i<k\}$ on domain $W_k$.

\begin{prop}\label{prop.ex}
Let $\mathcal{U}$ and $\mathcal{V}$ be any nonprincipal ultrafilters and let $\pi:\om\times\om\ra\om$ be defined by $\pi(m,n)=m$.
Let $f:2^{\om\times\om}\ra 2^{\om}$ be defined by $f(A)=\pi[A]$.
Then for any cofinal subset $\mathcal{C}$ of $\mathcal{U}\cdot\mathcal{V}$,
$f\re \mathcal{C}$ is not 
generated by a monotone basic map in the sense of Definition \ref{defn.basic}:
There is no strictly increasing sequence $(k_m)_{m<\om}$ and monotone basic map $\hat{f}:\bigcup_{m<\om}2^{W_{k_m}}\ra 2^{<\om}$
generating $f\re \mathcal{C}$.
\end{prop}

\begin{proof}
Let $g$ denote the restriction of $f$ to the ultrafilter $\mathcal{U}\cdot\mathcal{V}$.
Notice that $g$ is  a monotone cofinal map onto $\mathcal{U}$.
Suppose toward a contradiction that there is a cofinal subset $\mathcal{C}$ of $\mathcal{U}\cdot\mathcal{V}$ for which
$g\re \mathcal{C}$ is represented by a monotone basic function.
In this context, using the Cantor topology on $2^{\om\times\om}$ in place of $2^{\om}$,
 Definitions
\ref{def.levinitpres} and \ref{defn.basic}
are interpreted as follows:
There is a strictly increasing sequence $(k_m)_{m<\om}$
such that letting $C=\{X \re W_{k_m}:X\in\mathcal{C},\ m<\om\}$,
there is a monotone  basic map  $\hat{g}:C\ra 2^{<\om}$ such that for each $X\in\mathcal{C}$,
\begin{equation}
g(X)=\bigcup_{m<\om}\hat{g}(X \re W_{k_m}).
\end{equation}
To be clear in this context,
$\hat{g}$ is {\em level preserving} means that for each $s\in C\cap 2^{W_{k_m}}$, $\hat{g}(s)$ is a member of $2^m$.
For two members $s$ and $t$ of $C$,
$t$ end-extends $s$, written $s\sqsubseteq t$, if and only if the domain of $s$ is $W_j$ and  the domain of $t$ is $W_l$, for some
$j\le l$,
and for each $i<j$, $s(w_i)=t(w_i)$.

By the Extension Lemma \ref{thm.PropACtsMaps}, there is a monotone map $\tilde{f}:2^{\om\times\om}\ra 2^{\om}$ which is represented by a monotone basic map $\hat{f}$ on $\bigcup_{m<\om}2^{W_{k_m}}$ such that $\tilde{f}\re\mathcal{C}=g\re\mathcal{C}$.
As seen in the 
 proof of Lemma \ref{thm.PropACtsMaps}, modified to the current  context, 
the map $\hat{f}$ is  defined by
\begin{equation}\label{eq.4.2}
\hat{f}(t)(w_i)=\max\{\hat{g}(s)(w_i):s\in C,\  |s|\le k_m \mathrm{\ and \ } [[s]]\sse[[t]]\},
\end{equation}
for $t\in 2^{W_{k_m}}$ and $i<m$.

\begin{claim}\label{claim.ex4}
There is a $j<\om$ and an infinite collection $\{s_l:l<\om\}\sse C$,
satisfying the following:
For each $l<\om$, letting $d_l$ denote $[[s_l]]$,
 $j=\min(\pi[d_l])$,
and
letting $i_l$ denote the least $i$ such that both $w_i\in d_l$ and $j=\pi(w_i)$,
$(i_l)_{l<\om}$ forms a strictly increasing sequence.
\end{claim}

\begin{proof}
For $s\in C$, let $d_s$ denote $[[s]]$.
Suppose toward a contradiction that 
 for each $j<\om$,
 there is an $\tilde{i}_j$
 such that for each $s\in C$ with $j=\min(\pi[d_s])$,
 there is an $i<\tilde{i}_j$ such that $w_i\in d_s$ and $j=\pi(w_i)$.
Fix $X\in\mathcal{C}$.
Since $\mathcal{C}$ is a cofinal subset of the Fubini product of two nonprincipal ultrafilters,
it follows that
the set 
\begin{equation}
J=\{\min(\pi[d_s]):s\in C\mathrm{\ and\ }d_s\sse X\}
\end{equation}
is infinite.
For $j\in J$  let 
\begin{equation}
S_j=\{s\in C:d_s\sse X\mathrm{\ and\ }\min(\pi[d_s])=j\}.
\end{equation}
Then for each $j\in J$ and $s\in S_j$,
 there is  some $w_{i}\in d_s$ with $i<\tilde{i}_j$.
Thus, every member of $\mathcal{C}$ has infinite intersection with $\{w_i: \exists j\, (j=\pi(w_i)\mathrm{\ and\ }i<\tilde{i}_j)\}$.
 This contradicts the fact that $\mathcal{V}$ is nonprincipal.

 Thus, the negation of the supposition holds:
 There is a $j<\om$ such that for each $\tilde{i}<\om$, there is an $s\in C$ with $\min(\pi[d_s])=j$ such that, whenever $w_i\in d_s$ and $j=\pi(w_i)$, then $i\ge \tilde{i}$.
 Fix such a $j$.
 Take $s_0\in C$ such that $\min(\pi[d_{s_0}])=j$.
 Take $i_0$ least such that 
 $w_{i_0}\in d_{s_0}$ and 
  $\pi(w_{i_0})=j$.
Using $i_0$ as the next $\tilde{i}$,
there is an $s_1\in C$ with $\min(\pi[d_{s_1}])=j$ and $i_1>i_0$,
where  $i_1$ is least such that $w_{i_1}\in d_{s_1}$ and $\pi(i_1)=j$.
In this way, one constructs a collection $\{s_l:l<\om\}$ satisfying the Claim.
\end{proof}

Take $j<\om$ and $\{s_l:l<\om\}$ as in the  Claim.
Define $Y_l\in \mathcal{U}\cdot\mathcal{V}$ by
$Y_l=s_l\cup (j,\om)\times\om$.
Then $Y_l\ra Y$, where
$Y= (j,\om)\times\om$, which is a member of $\mathcal{U}\cdot\mathcal{V}$.
Since $\tilde{f}$ is generated by a basic map, $\tilde{f}$ is continuous on $2^{\om\times\om}$.
Hence, $\tilde{f}(Y_l)$ converges to $\tilde{f}(Y)$.

On the other hand,  we shall show that 
for each $l<\om$,
$j$ is in $\tilde{f}(Y_l)$
while
 $j$ is not in $\tilde{f}(Y)$, contradicting continuity of $f$.
Note that for each $l<\om$,
 $s_l\in C$ implies there is an $X\in\mathcal{C}$ whose characteristic function extends $s_l$.
Since  $g$ is generated by the basic map $\hat{g}$, and $j\in  \pi[X]$  implies $j\in g(X)$,
it follows that
 $j\in [[\hat{g}(s_l)]]$.
 By the definition of $\hat{f}$  in (\ref{eq.4.2}), it follows that
 $j\in [[\hat{f}(s_l)]]\sse \tilde{f}(Y_l)$.
 However, $j$ is not in $\tilde{f}(Y)$,
 since $\tilde{f}$ is generated by $\hat{f}$,
it follows from  the definition of $\hat{f}$ in (\ref{eq.4.2}) that
$j\not\in [[\hat{f}(Y\re k_m)]]$ for any $m<\om$.

 Thus,  there is no cofinal $\mathcal{C}\sse\mathcal{U}\cdot\mathcal{V}$ for which $g\re\mathcal{C}$ is  generated by a monotone basic map on the topological space $2^{\om\times\om}$.
\end{proof}

However,
we will soon show
that each ultrafilter $\mathcal{W}$ which is an
iterated Fubini product of p-points has finitely generated Tukey reductions which, moreover, are basic, and hence continuous, with respect to the topology on  the appropriate tree space.
Toward this end, we proceed to give the definition of basic for this context, and then prove the main results of this section.

\begin{notn}\label{notn.nec}
For any subset $A\sse[\om]^{<\om}$, recall that $\hat{A}$ denotes the set of all initial segments of members of $A$.
For any front $B$,    we let $\hat{B}^-$ denote $\hat{B}\setminus B$.
For any subset $A\sse[\om]^{<\om}$ and $k<\om$,
let $A\re k$ denote $\{a\in A:\max(a)<k\}$.
For $A\sse\hat{B}$ and $k<\om$,
we shall abuse notation and also use 
$A\re k$ 
 to denote the characteristic function of the set $A\re k$ on domain $\hat{B}\re k$.
For each $k<\om$, let $2^{\hat{B}\re k}$ denote the
set of all functions from $\hat{B}\re k$ into $\{0,1\}$.
Notice that this is exactly the
 collection of characteristic functions of subsets of $\hat{B}\re k$ on domain $\hat{B}\re k$.
\end{notn}

\begin{defin}\label{defn.basicTredtree}
Let $B$ be a front on $\om$,  $\tilde{T}\sse \hat{B}$ be a tree, and
  $(n_k)_{k<\om}$ be
an increasing sequence.
We say that a function
$\hat{f}:\bigcup_{k<\om}2^{\tilde{T}\re n_k}\ra 2^{<\om}$ is
{\em level preserving} if  $\hat{f}: 2^{\tilde{T}\re {n_k}}\ra 2^{k}$, for each $k<\om$.
$\hat{f}$ is {\em end-extension preserving} if  whenever $k<m$ and $A \sse\tilde{T}$ 
  then $\hat{f}(A \re n_k)=\hat{f}(A\re n_m)\re k$.
$\hat{f}$ is {\em basic} if it is level and end-extensions preserving.
$\hat{f}$ is {\em monotone} if whenever  $A\sse C\sse\tilde{T}$ and $k<\om$,
 then  $[[\hat{f}(A \re n_k)]]\sse [[\hat{f}(C \re n_k)]]$.

Let $\mathcal{U}$ be an ultrafilter on $B$ generated by $(\mathcal{U}_c:c\in \hat{B}^-)$-trees, let $f:\mathcal{U}\ra\mathcal{V}$ be a monotone cofinal map, where $\mathcal{V}$ is an ultrafilter on base $\om$, and let $\tilde{T}\in\mathfrak{T}(\vec{\mathcal{U}})$.
Let $\mathfrak{T}\re\tilde{T}$ denote the set of all $\vec{\mathcal{U}}$-trees contained in $\tilde{T}$.
We say that
$\hat{f}:\bigcup_{k<\om}2^{\tilde{T}\re n_k}\ra 2^{<\om}$ {\em generates} $f$ on $\mathfrak{T}\re \tilde{T}$ if
for each   $T\in\mathfrak{T}\re\tilde{T}$,
\begin{equation}
f([T])=\bigcup_{k<\om} \hat{f}(T \re n_k).
\end{equation}

We say that  $\mathcal{U}$ has {\em basic Tukey reductions}
if whenever $f:\mathcal{U}\ra\mathcal{V}$ is  a monotone cofinal map,
then there is a $\tilde{T}\in\mathfrak{T}(\vec{\mathcal{U}})$ and a basic  map $\hat{f}$ which generates $f$ on $\mathfrak{T}\re \tilde{T}$.
\end{defin}

\begin{rem}\label{rem.basiccts}
Note that if $\hat{f}$ witnesses that $f$ is basic on $\mathfrak{T}\re\tilde{T}$,
then $\hat{f}$ generates a continuous map on the collection of {\em trees} in $\mathfrak{T}\re\tilde{T}$,  continuity being with respect to the Cantor topology on $2^{\hat{B}}$.
Moreover, we may define a  map $\hat{g}$ on $B$ as follows:
For each finite subset $A\sse B$, define $\hat{g}(A)=\hat{f}(\hat{A})$, where $\hat{A}$ is the collection of all initial segments of members of $A$.
Then $\hat{g}$ is finitary, but not necessarily continuous on $2^B$, and $\hat{g}$ generates $f$ on  $\{[T]:T\in\mathfrak{T}\re\tilde{T}\}$ which is a base for the ultrafilter.
Thus,  for ultrafilters generated by $\mathcal{U}$-trees,  basic Tukey reductions imply finitely represented Tukey reductions on the original base set $B$.
\end{rem}

Now we prove the main theorem of this section.
Fix a total order of $[\om]^{<\om}$ in order type $\om$ such that $\max a<\max b$ implies $a\prec b$
for all $a,b\in [\om]^{<\om}$.
Note that  for each $k<\om$, the set $\{c\in[\om]^{<\om}:\max   c=k  \}$ 
forms a finite interval in $([\om]^{<\om},\prec)$.

\begin{thm}\label{thm.allFubProd_p-point_cts}
Let $B$ be any  front and $\vec{\mathcal{U}}=(\mathcal{U}_c:c\in \hat{B}^-)$ be a sequence of p-points.
Then the ultrafilter $\mathcal{U}$ on base $B$ generated by the $\vec{\mathcal{U}}$-trees
has basic Tukey reductions.
Therefore, every countable iteration of Fubini products of p-points has monotone basic Tukey reductions.
\end{thm}

\begin{proof}
Let $\mathcal{V}$ be some ultrafilter Tukey reducible to
$\mathcal{U}$, and let $f:\mathcal{U}\ra\mathcal{V}$ be a monotone cofinal map.
If $\mathcal{V}$ is a principal ultrafilter, say generated by the singleton $\{r\}$, then we claim that the theorem trivially holds.
Since the set consisting of $\{r\}$ is cofinal in $\mathcal{V}$, there is some $X\in\mathcal{U}$ such that $f(X)=\{r\}$.
Set $n_k=k$ and define
$\hat{f}:\bigcup_{k<\om}2^{\hat{B}\re k}\ra 2^{<\om}$
as follows:
For $k<\om$ and  $s\in 2^{\hat{B}\re k}$, $\hat{f}(s)$ is the sequence  in $2^k$  such that for each $i<k$, $\hat{f}(s)(i)=0$ if and only if $i\ne r$.
Then $\hat{f}$ is monotone basic and generates $f$ on $\mathcal{U}\re X$.
Thus, we shall assume from now on that $\mathcal{V}$ is nonprincipal.

We let $\mathfrak{T}$ denote $\mathfrak{T}(\mathcal{U}_c:c\in \hat{B}^-)$,
 the set of all $\vec{\mathcal{U}}$-trees.
Recall that $\mathfrak{T}$ is a base for the ultrafilter $\mathcal{U}$.
We make the convention that $\max \emptyset=-1$.
For each $k<\om$, let $\hat{B}\re k$ denote the collection of all $b\in \hat{B}$ with $\max b<k$.
Thus, $\hat{B}\re 0=\{\emptyset\}$, $\hat{B}\re 1=\{\emptyset,\{0\}\}$, and so forth.
Fix an enumeration of  the finite, non-empty $\sqsubseteq$-closed  subsets of $\hat{B}$ as $\lgl A_i:i<\om\rgl$
so that for each $i<j$, $\max\bigcup A_i\le\max\bigcup A_j$.
Let $(p_k)_{k<\om}$ denote the strictly increasing sequence
 so that for each $k$,
the sequence $\lgl A_i:i<p_k\rgl$ lists all $\sqsubseteq$-closed subsets of $\hat{B}\re k$.
(For example, if $B=[\om]^2$, then $\hat{B}=[\om]^{\le 2}$ and $\hat{B}^-=[\om]^{\le 1}$, and we may let
$A_0=\{\emptyset\}$, $A_1=\{\emptyset, \{0\}\}$,
$A_2=\{\emptyset, \{0\},\{1\}\}$,
$A_3=\{\emptyset,\{0\},\{0,1\}\}$,
$A_4=\{\emptyset,\{0\},\{0,1\},\{1\}\}$,
$A_5=\{\emptyset,\{1\}\}$.
Note that $p_0=1$, $p_1=2$, and $p_2=6$.)

For $k<\om$ and $i<p_k$, define
\begin{equation}\label{eq.20}
\hat{B}^k_i=
A_i\cup
\{b\in\hat{B}:\exists a\in A_i(b\sqsupset a \mathrm{\ and\ }\min(b\setminus a)\ge k)\}.
\end{equation}
Thus, $\hat{B}^k_i$ is the maximal tree in $\mathfrak{T}$ for which $T\re k=A_i$.
For a tree $T\sse\hat{B}$ and $c\in T\cap \hat{B}^-$,
define the notation
\begin{equation}
U_c(T)=\{l>\max(c):c\cup\{l\}\in T\}.
\end{equation}
 We refer to $U_c(T)$ as {\em the set of immediate extensions of $c$ in $T$}.
Note  that if $T\in\mathfrak{T}$, then  for each $c\in T\cap \hat{B}^-$, $U_c(T)$ is a member of  $\mathcal{U}_c$.
For $c\in \hat{B}^-$, recall that $\hat{B}_c$ denotes the tree of all $a\in\hat{B}$ such that either $a\sqsubseteq c$ or else $a\sqsupset c$.

Our  goal is to construct a tree $\tilde{T}\in\mathfrak{T}$ and find a sequence $(n_k)_{k<\om}$ of good cut-off points such that the following  $(\circledast)$ holds.
\begin{enumerate}
\item[$(\circledast)$]
For each $T\sse\tilde{T}$ in $\mathfrak{T}$,  $k<\om$, and $i<p_{n_k}$ such that $A_i=T\re n_k$, for every $j\le k$,
$$
j\in f([T]) \Llra j\in f([\tilde{T}\cap \hat{B}^{n_k}_i]).
$$
\end{enumerate}

\begin{claim1}\label{clm.circldastgivesthm}
The property  $(\circledast)$ implies that $f$  has  monotone  basic Tukey reductions.
\end{claim1}

\begin{proof}
For $k<\om$ and $T\in\mathfrak{T}\re \tilde{T}$,
define
\begin{equation}
\hat{f}(T\re n_k)=f([\tilde{T}\cap \hat{B}^{n_k}_i])\re k,
\end{equation}
where $i$ is {\em the} integer below $p_{n_k}$ such that $T\re n_k=A_i$.
By definition, $\hat{f}$ is level preserving.
Let $l>k$ and  let $m$ be such that $T\re n_l=A_m$.
Then $A_m\re n_k=A_i$.
For $j< k$,
 $(\circledast)$ implies that
$j\in f([\tilde{T}\cap \hat{B}^{n_k}_i])$ if and only if
$j\in f([T])$
if and only if
$j\in f([\tilde{T}\cap \hat{B}^{n_l}_m])$.
Thus, $j\in [[\hat{f}(T\re n_k)]]$ if and only if $j\in [[\hat{f}(T\re n_l)]]$.
Therefore, $\hat{f}$ is end-extension preserving;
that is, $\hat{f}(T\re n_l)\re k=\hat{f}(T\re n_k)$.
Furthermore, $f([T])=\bigcup_{k<\om}\hat{f}(T\re n_k)$;
thus, $\hat{f}$ generates $f$ on $\mathfrak{T}\re \tilde{T}$.
To see that $\hat{f}$ is monotone, suppose that $S\re n_k\sse T\re n_k$ for some $S,T\in \mathfrak{T}\re \tilde{T}$.
Let $i,j<p_k$ be such that $A_i=S\re n_k$ and $A_j=T\re n_k$.
Then 
\begin{equation}
\hat{f}(S\re n_k)=f([\tilde{T}\cap \hat{B}^{n_k}_i])\re k \sse f([\tilde{T}\cap \hat{B}^{n_k}_j])\re k=\hat{f}(T\re n_k),
\end{equation}
where $\sse$ holds because of $(\circledast)$ and the fact that $f$ is monotone and $\tilde{T}\cap \hat{B}^{n_k}_i\sse \tilde{T}\cap \hat{B}^{n_k}_j$.
Therefore, $\hat{f}$ is a monotone basic map  generating $f$  on $\mathfrak{T}\re\tilde{T}$.
\end{proof}

The construction of $\tilde{T}$ and $(n_k)_{k<\om}$ takes place in three stages.
\vskip.1in

\noindent \bf Stage 1. \rm
In the first stage toward the construction of $\tilde{T}$,
we will choose some $R^k_i\in\mathfrak{T}$ with $A_i\sse R_i^k$
such that  for all  $k<\om$, the following  holds:
\begin{enumerate}
\item[$(*)_k$]
For all  $i<p_k$ and  $T\sse R^k_i$  in $\mathfrak{T}$ with $T\re k=A_i$,
for each $j\le k$,
 $j\in f([T])\Llra j\in f([R^k_i])$.
\end{enumerate}

We point out  that for any  front,
 $A_0$ is always $\{\emptyset\}$ and $p_0=1$.
Since we are assuming $\mathcal{V}$ is nonprincipal,
choose  an  $R^0_0$ in $\mathfrak{T}$ so that $0\not\in f([R^0_0])$.
Now let  $k>0$, and suppose we have chosen $R^{l}_j$ for all $l<k$ and $j<p_{l}$.
For
 $i<p_{k-1}$, if there is an  $R\sse R^{k-1}_i$  in $\mathfrak{T}$ such that $R\re k=A_i$ and
$k\not\in f([R])$, then let $R^k_i$ be such an $R$;
if not, let $R^k_i=R^{k-1}_i\cap \hat{B}^k_i$.
Now suppose that $p_{k-1}\le i<p_k$.
If there is an $R\in\mathfrak{T}$ such that  $R\re k=A_i$
and $0\not \in f([R])$,  let $R^k_{i,0}$ be such an $R$;
if not,  let $R^k_{i,0}=\hat{B}^k_i$.
Given $R^k_{i,j}$ for $j<k$,
if there is an $R\in\mathfrak{T}$ such that $R\sse R^k_{i,j}$, $R\re k=A_i$,
and $j+1\not \in f([R])$, then let $R^k_{i,j+1}$ be such an $R$;
if not, then let $R^k_{i,j+1}=R^k_{i,j}$.
Finally, let $R^k_i=R^k_{i,k}$.

It follows from the construction that for all $1\le l\le k$ and $p_{l-1}\le i<p_l$,
\begin{equation}
R^l_{i,0}\contains R^l_{i,1}\contains\dots\contains R^l_{i,l}=R^l_i\contains  R^k_i,
\end{equation}
and moreover,    for any $j\le l$,
\begin{equation}
R^l_{i,j}\re l=R^l_i\re l=R^k_i\re k= A_i.
\end{equation}

Fix $k<\om$:
we  check  that $(*)_k$ holds.
Let  $i<p_k$, $T\sse R^k_i$ in $\mathfrak{T}$ with $T\re k=A_i$, and $j\le k$ be given.
If  $j\in f([T])$, then
$j$ must be in
$f([R^k_i])$, since $T\sse R^k_i$ and $f$ is monotone.
Now suppose that $j\not\in f([T])$; we will show that $j\not\in f([R^k_i])$.
Let $p_{-1}=0$,
and let $l\le k$ be the integer satisfying $p_{l-1}\le i<p_l$.
Note that  $\max\bigcup A_i=l-1$.
Thus,  $T\re k= A_i=T\re l$.
We now have two cases to check.

Case 1: $j\le l$.
Notice that $T\sse R^k_i\sse R^l_i\sse R^l_{i,j}$.
If $j=0$, then $R^l_{i,j}\sse \hat{B}^l_i$ and we let $R'$ denote $\hat{B}^l_i$;
if $j>0$, then $R^l_{i,j}\sse R^l_{i,j-1}$ and we let $R'$ denote $R^l_{i,j-1}$.
Since $j$ is not in  $f([T])$ and $T\re l=A_i$,
$T$ is a witness that there is an $R\sse R'$ with $R\re l=A_i$ such that $j\not\in f([R])$.
Thus, $R^l_{i,j}$ was chosen so that $j\not \in f([R^l_{i,j}])$.
It follows that $j\not\in f([R^k_i])$, since $R^k_i\sse R^l_{i,j}$ and $f$ is monotone.

Case 2:  $l<j\le k$.
In this case, $T\sse R^k_i\sse R^j_i\sse R^{j-1}_i$.
Since $T$ is a witness that there is an
$R\sse R^{j-1}_i$ with $R\re l=A_i$ and $j\not\in f([R])$,
$R^j_i$ was chosen so that $j\not\in f([R^j_i])$.
Thus, $j\not\in f([R^k_i])$, since $R^k_i\sse R^j_i$ and $f$ is monotone.

Therefore, $j\in f([T])$ if and only if $j\in f([R^k_i])$;
hence $(*)_k$ holds.
This concludes  Stage 1 of our construction.
\vskip.1in

Given $k<\om$ and $c\in \hat{B}^-\re k$,  define
\begin{equation}\label{eq.defSnkc}
S^k_c=\bigcap\{R^l_i: l\le k,\ i<p_l,\ \mathrm{and}\ c\in R^l_i\}.
\end{equation}
We claim that $S^k_c$ is a member of $\mathfrak{T}$  and that $c\in S^k_c$.
To see this, notice that
for  $c\in \hat{B}^-\re k$, letting $l\le k$ be least such that $l>\max c$,
then $c$
 is in $A_i$ for at least one $i<p_l$.
Since $A_i\sse R^l_i$, the set
$\{(l,i): l\le k,\ i<p_l,\ \mathrm{and}\ c\in R^l_i\}$ is nonempty;
hence, $c\in S^k_c$.
Moreover, $S^k_c$ is a member of $\mathfrak{T}$, since it is a finite intersection of members of $\mathfrak{T}$.
It follows that for each $a\in S^k_c\cap \hat{B}^-$, the set of $\{l>\max(a):a\cup\{l\}\in S^k_c\}$ is a member of the ultrafilter $\mathcal{U}_c$.
Define
\begin{equation}\label{eq.defUkc}
U^k_c:=U_c(S^k_c)=\{l>\max c:c\cup\{l\}\in S^k_c\}.
\end{equation}
Thus, for  each $c\in \hat{B}^-$ and $j=\max(c)+1$, we have
$S^j_c\contains S^{j+1}_c\contains\dots$, each of which is a member of $\mathfrak{T}$;
and
$U_c^j\contains U_c^{j+1}\contains\dots$, each of which is a member of the p-point  $\mathcal{U}_c$.
\vskip.1in

\bf Stage 2. \rm
In this stage we  construct a tree $T^*$  in $\mathfrak{T}$
which will be thinned down one more time in Stage 3 to obtain a subtree $\tilde{T}\sse T^*$ in $\mathfrak{T}$ such  that $f\re\mathfrak{T}\re\tilde{T}$ is basic.
The tree $T^*$ which we construct in this stage
will have sets of immediate successors
\begin{equation}\label{eq.U_c}
U_c:=U_c(T^*)=\{l\ge \max(c) +1:c\cup\{l\}\in T^*\},
\end{equation}
for  $c\in T^*\cap \hat{B}^-$.
The sets $U_c$  will have  interval gaps which have right endpoints which line up often and in a useful way (meshing).
This will
 aid in  finding  good cut-off points $n_k$ needed  in Stage 3 to thin $T^*$ down to $\tilde{T}$.
Toward obtaining these interval gaps, we will construct  a family of  functions which we call  {\em meshing functions} $m(c,\cdot):\om\ra\om$ satisfying the following
`meshing property':
\begin{enumerate}
\item[$(\dag)$]
For each  $c\in \hat{B}^-$ and $j<\om$,
if $a\in\hat{B}^-$ is such that $a\prec c$,
then there exists $i<\om$ such that
 $m(a,2i)=m(c,2j)$.
\end{enumerate}
The meshing functions of $(\dag)$ will  aid in
 obtaining a tree $T^*\in\mathfrak{T}$ with the following properties:
\begin{enumerate}
\item[$(\ddag)$]
For all $c\in T^*\cap \hat{B}^-$,
\begin{enumerate}
\item[(a)]
$U_c\sse U_c^{\max(c)+1}$; and
\item[(b)]
For all  $i<\om$,
$U_c\setminus m(c,2i)
=U_c\setminus m(c,2i+1)
\sse U_c^{m(c,2i)}$.
\end{enumerate}
\end{enumerate}

We now begin the construction of the meshing functions $m(c,\cdot)$ and the sets $U_c$,
 proceeding by recursion on the well-ordering $(\hat{B}^-,\prec)$.
Since  $\emptyset$ is $\prec$-minimal in $\hat{B}^-$, we start by choosing
$g_{\emptyset}$, $m(\emptyset,\cdot)$, and $Y_{\emptyset}$   as follows.
Since $\mathcal{U}_{\emptyset}$ is a p-point,
we may choose a $U_{\emptyset}^*\in\mathcal{U}_{\emptyset}$ such that $U_{\emptyset}^*\sse^* U_{\emptyset}^k$ for each $k$.
(Recall the definition of $U_c^k$ from equation (\ref{eq.defUkc}).)
Let  $g_{\emptyset}:\om\ra\om$  be a
 strictly increasing function such that  for each $k$,
$U_{\emptyset}^*\setminus g_{\emptyset}(k+1)\sse U_{\emptyset}^{g_{\emptyset}(k)}$,
and $g_{\emptyset}(0)>0$.
If $\bigcup_{i\in\om}[g_{\emptyset}(2i),g_{\emptyset}(2i+1)) \in\mathcal{U}_{\emptyset}$, then
define $m(\emptyset,k)=g_{\emptyset}(k+1)$;
otherwise,
$\bigcup_{i\in\om}[g_{\emptyset}(2i+1),g_{\emptyset}(2i+2))\in\mathcal{U}_{\emptyset}$,
and we define $m(\emptyset,k)=g_{\emptyset}(k)$.
Let $Y_{\emptyset}=\bigcup_{i\in\om}[m(\emptyset,2i+1),m(\emptyset,2i+2))$ and
define
\begin{equation}
U_{\emptyset} =U_{\emptyset}^0\cap U_{\emptyset}^*\cap Y_{\emptyset}.
\end{equation}
Note that for each $k$, $U_{\emptyset}\setminus m(\emptyset,k+1)\sse U_{\emptyset}^{m(\emptyset,k)}$.

Now suppose $c\in \hat{B}^-$ and for all $b\prec c$ in $\hat{B}^-$, $g_b$ and $m(b,\cdot)$ have been defined.
Since $\mathcal{U}_{c}$ is a p-point, there is a $U_{c}^*\in\mathcal{U}_{c}$   for which $U_{c}^*\sse^*U_{c}^k$, for all $k>\max c$.
Let $a$ denote the immediate $\prec$-predecessor of $c$ in $\hat{B}^-$.
Let  $g_{c}:\om\ra\om$  be a
 strictly increasing function such that  $g_c(0)>\max \{\max c,g_a(2)\}$, and
\begin{enumerate}
\item[($1_{g_c}$)]
For each $i<\om$,
$U_{c}^*\setminus g_{c}(i+1)\sse U_{c}^{g_{c}(i)}$; and
\item[($2_{g_c}$)]
  For each $j<\om$, there is an $i>0$ such that
$g_{c}(j)=m(a, 2i)$.
\end{enumerate}
Let $Y_{c}$ denote the one of the two sets
$\bigcup_{i\in\om}[g_{c}(2i),g_{c}(2i+1))$
or
$\bigcup_{i\in\om}[g_{c}(2i+1),g_{c}(2i+2))$
which is in
$\mathcal{U}_{c}$.
In the first case define
define
$m(c,i)=g_{c}(i+1)$;
in the second case
$m(c,i)=g_{c}(i)$.
Then
\begin{equation}\label{eq.defYc}
Y_c=\bigcup_{i<\om}[m(c,2i+1),m(c,2i+2))
\end{equation}
and  is in $\mathcal{U}_c$.
Let
\begin{equation}\label{eq.U_c}
U_{c} =U_c^{\max(c)+1}\cap U_{c}^*\cap Y_{c}.
\end{equation}
This concludes the recursive definition.

We check that $(\dag)$ holds.
Let  $c\in \hat{B}^-$ and
let $a_0\prec \dots\prec a_l\prec c$ be the enumeration of all $\prec$-predecessors of $c$ in $\hat{B}^-$.
Let $j<\om$ be given.
Either $m(c,2j)=g_c(2j)$ or $m(c,2j)=g_c(2j+1)$.
By ($2_{g_c}$), $g_c(2j)=m(a_l,2i)$ for some $i$, and  $g_c(2j+1)=m(a_l,2i)$ for some $i$.
Thus, there is an $i$ such that $m(a_l,2i)=m(c,2j)$.
Let $i_l$ denote this $i$.
Likewise, either $m(a_l,2i_l)=g_{a_l}(2i_l)$ or $m(a_l,2i_l)=g_{a_l}(2i_l+1)$.
By ($2_{g_{a_l}}$), $g_{a_l}(2i_l)=m(a_{l-1},2i)$ for some $i$, and  $g_{a_l}(2i_l+1)=m(a_{l-1},2i)$ for some $i$.
Let $i_{l-1}$ denote the $i$ such that $m(a_{l-1},2i_{l-1})=m(a_l,2i_l)$.
Continuing in this manner, we obtain numbers $i_k$, $k\le l$, such that
\begin{equation}
m(c,2j)=m(a_l,2i_l)=m(a_{l-1},2i_{l-1})=\dots =m(a_0,2i_0).
\end{equation}
Hence,  $(\dag)$ holds.

\begin{claim2}\label{claim.referee}
There exists a strictly increasing sequence $(m_k)_{k<\om}$ such that, for all $k$,
\begin{equation}
\forall c\in \hat{B}^-\re m_k\  \exists r\, (m(c,2r)=m_{k+1}).
\end{equation}
\end{claim2}

\begin{proof}
Let $m_0$ be arbitrary, and let $c_0$ be the $\prec$-maximum of $\hat{B}^-\re m_0$.
Since $m(c_0,\cdot)$ is strictly increasing, we can fix $j_0\in\om$ such that $m(c_0,2j_0)>m_0$.
Let $m_1=m(c_0,2j_0)$.
In general, given $m_k$, let $c_k$ be the $\prec$-maximum of $\hat{B}^-\re m_k$.
Since $m(c_k,\cdot)$ is strictly increasing, we can fix  some $j_{k}<\om$ such that $m(c_k,2j_k)>m_k$, and let
 $m_{k+1}=m(c_k,2j_k)$.
In this manner, we inductively construct the sequence $(m_k)_{k<\om}$.
To check that this sequence has the desired property, let $k<\om$ be given and
fix $c\in\hat{B}^-\re m_k$.
Since $c\preccurlyeq c_k$,
it follows from $(\dag)$ that there exists an $r$ such that
$m(c,2r)=m(c_k,2j_k)=m_{k+1}$.
\end{proof}

Let  $T^*$ be the  tree in $\mathfrak{T}$ defined
by declaring for each $c\in \hat{B}^-\cap T^*$, $U_c(T^*)=U_c$.
If the reader is not satisfied with this top-down construction (which {\em is} precise as $\emptyset$ is in  every member of $\mathfrak{T}$ and this completely determines the rest of $T^*$), we point out that
$T^*$ can also be seen as being constructed
  level by level as follows.
Let $\emptyset\in T^*$,
and for each $l\in U_{\emptyset}$,
put $\{l\}$ in $T^*$, so that the first level of $T^*$ is exactly $\{\{l\}:l\in U_{\emptyset}\}$.
Suppose we have constructed the tree $T^*$ up to level $k$, meaning that we know exactly what $T^*\cap \hat{B}\cap[\om]^{\le k}$ is.
For each $c\in \hat{B}^-\cap T^*\cap[\om]^{k}$,
let the immediate successors of $c$ in $T^*$ be exactly the set $U_c$;
in other words, for each $l>\max c$, put $c\cup\{l\}\in T^*$ if and only if $l\in U_c$.
Recalling that $\max c<g_c(0)\le m(c,0)$ and $U_c\sse Y_c=Y_c\setminus m(c,0)$,
we see that each element of $U_c$ is strictly greater than $\max c$.
Hence, by constructing $T^*$ in this manner, we obtain a member of $\mathfrak{T}$ such that for each $c\in T^*\cap \hat{B}^-$, $U_c(T^*)$ is exactly $U_c$.

We now check that   $(\ddag)$ holds.
Let  $c\in T^*\cap \hat{B}^-$ be given.
By (\ref{eq.U_c}),
$U_c\sse U^{\max(c)+1}_c$,  so
 $(\ddag)$ (a) holds.
By equation (\ref{eq.defYc}),
 we see that  $Y_c\cap [m(c,2i),m(c, 2i+1))=\emptyset$ for each $i$.
Thus,
\begin{equation}\label{eq.dagb}
U_c\cap [m(c,2i),m(c,2i+1))=\emptyset,
\end{equation}
since $U_c\sse Y_c$ by (\ref{eq.U_c}).
Recall that $U^*_c$ diagonalizes the collection of sets $U_c^k$ for all $k>\max c$, and  the function $g_c$ was chosen to witness this diagonalization so that ($1_{g_c}$) holds.
Either
$m(c,i)=g_c(i)$ and $m(c,i+1)=g_c(i+1)$,
or else
$m(c,i)=g_c(i+1)$ and $m(c,i+1)=g_c(i+2)$.
In either case, ($1_{g_c}$) implies that $U^*_c\setminus
m(c,i+1)\sse U_c^{m(c,i)}$.
Thus
\begin{equation}\label{eq.dagc}
U_{c}\setminus m(c,i+1)\sse U_{c}^{m(c,i)},
\end{equation}
since $U_c\sse U^*_c$ by (\ref{eq.U_c}).
$(\ddag)$ (b) follows from (\ref{eq.dagb}) and (\ref{eq.dagc}).
This finishes Stage 2 of our construction.
\vskip.1in

\noindent\bf Stage 3. \rm
We will show that there is a strictly increasing sequence  $(n_k)_{k<\om}$ and a subtree $\tilde{T}\sse T^*$ in $\mathfrak{T}$ so that for all $k$,
\begin{equation}\label{eq.new4.13}
\forall c\in\tilde{T}\cap(\hat{B}^-\re n_k)\ \exists r\, (m(c,2r)=n_k).
\end{equation}
The  following lemma  uses induction on the rank of the front.

\begin{lem}\label{lem.hconditions.ref}
Let $B$ be a front on $\om$, and let $T^*$ be a $\vec{\mathcal{U}}$-tree,
where $\vec{\mathcal{U}}=\lgl\mathcal{U}_c:c\in \hat{B}^-\rgl$
is a sequence of non-principal ultrafilters on $\om$.
Suppose that strictly increasing functions $m(c,\cdot):\om\ra\om$ for every $c\in\hat{B}^-$ and a strictly increasing sequence $(m_k)_{k<\om}$ are given such that,
for all $k$,
\begin{equation}\label{eq.oldm_k}
\forall c\in\hat{B}^-\re m_k\ \exists r\, (m(c,2r)=m_{k+1}).
\end{equation}
Then there exist $\tilde{T}\sse T^*$ such that $\tilde{T}$ is a $\vec{\mathcal{U}}$-tree and a subsequence $(n_k)_{k<\om}$ of $(m_k)_{k<\om}$
such that, for all $k$,
\begin{equation}\label{eq.n_kproperty}
\forall c\in\tilde{T}\cap(\hat{B}^-\re n_k)\  \exists r \, (m(c,2r)=n_k).
\end{equation}
\end{lem}

\begin{proof}
The proof will be by induction on the rank of $B$.
First assume rank$(B)=1$.
In this case, $B=\hat{B}=\{\emptyset\}$ and $\hat{B}^-=\{\}$.
Also notice that the only $\lgl\rgl$-tree is $\{\emptyset\}$.
Therefore, setting $\tilde{T}=\{\emptyset\}$ and $(n_k)_{k<\om}=(m_k)_{k<\om}$ will work.

Now assume that rank$(B)>1$.
Then, for every $l<\om$, $B_{\{l\}}$ is  a front on $[l+1,\om)$ of rank strictly less than rank$(B)$.
Given $l<\om$, set $m_l(c,\cdot)=m(\{l\}\cup c,\cdot)$ for $c\in {\hat{B}_{\{l\}}}^-$.
For $l\in U_{\emptyset}(T^*)$,
define
\begin{equation}\label{eq.T^*_l}
T^*_l=\{c\in \hat{B}_{\{l\}}:\{l\}\cup c\in T^*\},
\end{equation}
and observe that $T_l^*$ is a $\vec{\mathcal{U}}_l$-tree, where $\vec{\mathcal{U}}_l=\lgl\mathcal{U}_{\{l\}\cup c}:c\in{\hat{B}_{\{l\}}}^-\rgl$.

Next, we will inductively define subsequences $(n_k^j)_{k<\om}$ of $(m_k)_{k<\om}$ for every $j<\om$.
We will also make sure that $(n^h_k)_{k<\om}$ is a subsequence of $(n^j_k)_{k<\om}$ whenever $h\ge j$.
Furthermore, we will obtain a $\vec{\mathcal{U}}_l$-tree $\tilde{T}_l\sse T^*_l$ for every $l\in U_{\emptyset}(T^*)$, such that, for all $k$,
\begin{equation}\label{eq.ref.1}
\forall c\in \tilde{T}_l\cap ({\hat{B}_{\{l\}}}^-\re n^l_k)\
\exists r \, (m_l(c,2r)=n^l_k).
\end{equation}
Let $(m^0_k)_{k<\om}=(m_k)_{k<\om}$.
If $0\not\in U_{\emptyset}(T^*)$, simply let $(n^0_k)_{k<\om}=(m^0_k)_{k<\om}$.
If $0\in U_{\emptyset}(T^*)$,
apply the induction hypothesis to $B_{\{0\}}$ and $T_0^*$, with respect to the functions $m_0(c,\cdot)$ and the sequence $(m_k^0)_{k<\om}$.
This will yield a $\vec{\mathcal{U}}_0$-tree $\tilde{T}_0\sse T^*_0$ and a subsequence $(n^0_k)_{k<\om}$ of $(m^0_k)_{k<\om}$.

Let $(m^1_k)_{k<\om}=(n^0_k)_{k<\om}$.
If $1\not\in U_{\emptyset}(T^*)$, simply let $(n^1_k)_{k<\om}=(m^1_k)_{k<\om}$.
If $1\in U_{\emptyset}(T^*)$, apply the induction hypothesis to $B_{\{1\}}$ and $T_1^*$, with repsect to the functions $m_1(c,\cdot)$ and the sequence $(m_k^1)_{k<\om}$.
Continue 
 as in these first two steps to complete the inductive construction.

Set $h(0)=n^0_0$.
Given $h(k)$, fix $q(k)$ such that $n^{h(k)}_{q(k)}>h(k)$,
then set $h(k+1)=n^{h(k)}_{q(k)}$.
Notice that exactly one of
$\bigcup_{k<\om}[h(2k),h(2k+1))$ and
$\bigcup_{k<\om}[h(2k+1),h(2k+2))$
will belong to $\mathcal{U}_{\emptyset}$,
and denote it by $Z$.
Set $U_{\emptyset}(\tilde{T})=U_{\emptyset}(T^*)\cap Z$,
then define
\begin{equation}\label{eq.ref2}
\tilde{T}=\{\emptyset\}\cup\bigcup_{l\in U_{\emptyset}(\tilde{T})}\{\{l\}\cup c:c\in\tilde{T}_l\}.
\end{equation}
It is straightforward to check that $\tilde{T}$ is a $\vec{\mathcal{U}}$-tree contained in $T^*$.
If $Z=\bigcup_{k<\om}[h(2k),h(2k+1))$, then define $n_k=h(2k+2)$,
otherwise define $n_k=h(2k+1)$.
We will only complete the proof in the first case, as the other case is similar.

Fix $k<\om$ and $c\in \tilde{T}\cap(\hat{B}^-\re n_k)$.
If $c=\emptyset$, then $c\in \hat{B}^-\re m_{j-1}$, where $j>0$ is such that $m_j=n_k$.
Therefore, there exists $r$ such that $m(c,2r)=m_j=n_k$.
Now assume that $c\ne\emptyset$, and let $l=\min c$.
Since $l\in Z$ and $l<h(2k+2)$, we must have $l<h(2k+1)$.
In particular, $(n^{h(2k+1)})_{k<\om}$ is a subsequence of $(n^l_k)_{k<\om}$.
Therefore, there exists $q$ such that
$n^{h(2k+1)}_{q(2k+1)}=n^l_q$.
Finally, since $c\setminus \{l\}\in\tilde{T}_l\cap (\hat{B}_{\{l\}}^-\re n_q^l)$,
we see that there exists $r$ such that
\begin{equation}\label{eq.ref3}
m(c,2r)=m_l(c\setminus\{l\},2r)=n^l_q=n^{h(2k+1)}_{q(2k+1)}=h(2k+2)=n_k,
\end{equation}
which is what we needed to show.
\end{proof}
Taking $\tilde{T}$ as in Lemma \ref{lem.hconditions.ref} 
concludes Stage 3 of the construction.
\vskip.1in

Finally, we check that $(\circledast)$ holds.
Toward this, we first show that for all $k<\om$ and $i<p_{n_k}$, $\tilde{T}\cap\hat{B}^{n_k}_i\sse R^{n_k}_i$.
It will follow that
for each $T\in\mathfrak{T}\re\tilde{T}$ with $T\re n_k=A_i$, we in fact have $T\sse R^{n_k}_i$.
This along with
$(*)_{n_k}$ for all $k<\om$ will yield $(\circledast)$.

\begin{claim3}\label{claim.checkingcircledastom2}
Let  $k<\om$ and $i<p_{n_k}$   be given, and suppose that $A_i\sse \tilde{T}$.
Then $\tilde{T}\cap\hat{B}^{n_k}_i\sse R^{n_k}_i$.
\end{claim3}

\begin{proof}
Let $Q$ denote $\tilde{T}\cap\hat{B}^{n_k}_i$.
Since $A_i\sse \tilde{T}$, we see that $Q\re n_k=A_i$ which equals $R^{n_k}_i\re n_k$.
Thus, to prove the claim it is enough to show
 that
 for each $c\in Q\cap \hat{B}^-$,
\begin{equation}\label{eq.goaldast}
U_c(Q)\setminus n_k\sse U_c(R^{n_k}_i).
\end{equation}
Since $Q\sse\tilde{T}\sse T^*$, we see that
 for all $c\in Q\cap\hat{B}^-$,
\begin{equation}\label{eq.plus}
U_c(Q)\setminus n_k
\sse U_c(\tilde{T})\setminus n_k
\sse U_c(T^*)\setminus n_k.
\end{equation}
We have two cases for $c$.

Case 1: \rm $c\in Q\cap (\hat{B}^-\re n_k)$.
Then by   Lemma \ref{lem.hconditions.ref},
there is an $r$ such that $m(c,2r)=n_k$.
By  $(\ddag)$ (b),
we have that
\begin{equation}\label{eq.times}
U_c(T^*)\setminus n_k=
U_c(T^*)\setminus m(c, 2r +1)\sse U_c^{m(c, 2r)} =U_c^{n_k}.
\end{equation}
Since $Q\re n_k=A_i\sse R^{n_k}_i$,
$c$ is in $R^{n_k}_i$,
and hence, $S_c^{n_k}\sse R^{n_k}_i$,
recalling
(\ref{eq.defSnkc}).
Therefore, recalling (\ref{eq.defUkc}),
\begin{equation}\label{eq.squaredot}
U_c^{n_k}=U_c(S^{n_k}_c)\sse U_c( R_i^{n_k}).
\end{equation}
Hence, by (\ref{eq.plus}), (\ref{eq.times}), and (\ref{eq.squaredot}), we see that
\begin{equation}
U_c(Q)\setminus n_k\sse  U_c(T^*)\setminus n_k\sse U_c^{n_k} \sse U_c(R^{n_k}_i).
\end{equation}

Case 2: \rm 
For  $c\in Q\cap \hat{B}^-$ such that $\max c\ge n_k$, 
it follows from  $(\ddag)$ (a) that 
$U_c(T^*)\sse U_c(S^{\max(c)+1}_c)$.
The proof will proceed by induction on the cardinality of $c\setminus n_k$.

Suppose $|c\setminus n_k|=1$.
Let $l=\max c$, and let $a$ denote $c\setminus\{l\}$.
Then $l \in U_a(Q)\setminus n_k$.
Further,
$a$ is a member of $R^{n_k}_i$, since $a\in Q\re n_k=A_i\sse R^{n_k}_i$.
Since
by Case 1, $U_a(Q)\setminus n_k$ is contained in $U_a(R^{n_k}_i)$,
we have that $c\in R^{n_k}_i$.
Further,
\begin{equation}\label{eq.plusdiamond}
U_c(Q)\sse U_c(T^*) \sse U_c(S_c^{l+1})\sse U_c( R^{n_k}_i).
\end{equation}
since $l+1>n_k$, $i<p_{n_k}$ and $c\in R^{n_k}_i$ imply that $S^{l+1}_c\sse R^{n_k}_i$,
by the definition  (\ref{eq.defSnkc}) of $S^{l+1}_c $.
Thus, Case 2 holds for the basis of our induction scheme.

Now assume that  Case 2 holds for all $c\in Q\cap \hat{B}^-$ with $1\le |c\setminus n_k|\le m$.
Suppose $c\in Q\cap \hat{B}^-$ with $|c\setminus n_k|=m+1$.
Letting $l=\max c$ and  $a=c\setminus \{l\}$, the induction hypothesis applied to $a$ yields that $a\in R^{n_k}_i$ and
$U_a(Q)\sse U_a(R^{n_k}_i)$.
Thus, $c\in R^{n_k}_i$.
Again, as in  (\ref{eq.plusdiamond}),
we find that $U_c(Q)\sse U_c(R^{n_k}_i)$, which finishes the proof of Case 2.
\end{proof}

To finish the proof of the theorem, we prove that $(\circledast)$ holds.
Let  $T\in\mathfrak{T}\re \tilde{T}$, $k<\om$, and suppose $i<p_{n_k}$ is the integer such that   $T\re n_k=A_i$.
Letting $Q$ denote $\tilde{T}\cap\hat{B}^{n_k}_i$,
we see that $T\sse Q$.
By Claim 3, $Q\sse R^{n_k}_i$; so
 for all $j\le k$,
\begin{equation}\label{eq.TR}
j\in f([T])\Longleftrightarrow j\in f([R^{n_k}_i])
\end{equation}
by $(*)_{n_k}$.
Since  $Q\re n_k=A_i$, it follows  from  $(*)_{n_k}$  that
for all $j\le k$,
\begin{equation}\label{eq.TQR}
j\in f([Q])\Longleftrightarrow j\in f([R^{n_k}_i]).
\end{equation}
Equations (\ref{eq.TR}) and (\ref{eq.TQR}) complete the proof of $(\circledast)$.
By Claim 1, $f$ is monotone basic on $\mathfrak{T}\re\tilde{T}$.
The concludes the proof of the theorem.
\end{proof}

We conclude this section pointing out how $f$ may be basic on $2^{\hat{B}}$ while being only finitely generated on $2^B$.
Given a front $B$ and a set $\mathcal{C}\sse 2^B$,
 let $\hat{\mathcal{C}}$ denote $\{\hat{X}:X\in\mathcal{C}\}$, a subset of $2^{\hat{B}}$.
Letting $C$ denote $\{\hat{X}\re k_m:X\in\mathcal{C}$ and $m<\om\}$,
we point out that
any finitary function $\hat{f}:C\ra 2^{<\om}$
 determines
 functions $f':\mathcal{C}\ra 2^{\om}$
and $f^*:\hat{\mathcal{C}}\ra 2^{\om}$
by setting $ f'(X)=f^*(\hat{X})=\bigcup_{m<\om}\hat{f}(\hat{X}\re k_m)$.
In particular, $f'(X)=f^*(\hat{X})$ for each $X\in\mathcal{C}$.
The following is straightforward to prove.

\begin{prop}\label{prop.basiconBhatgivescts}
Suppose $B$ is a front, $\mathcal{C}$ is a subset of $2^B$, and  $C=\{\hat{X}\re k_m:X\in\mathcal{C},\ m<\om\}$.
If $\hat{f}:C\ra 2^{<\om}$ is a   basic map,
then $f^*$ is continuous on $\hat{\mathcal{C}}$ as a subspace of $2^{\hat{B}}$.
\end{prop}

In particular, given a map $f:\mathcal{U}\ra\mathcal{V}$ in the setting of Theorem \ref {thm.allFubProd_p-point_cts},
the map $f^*:\mathfrak{T}\re\tilde{T}\ra\mathcal{V}$ defined by
$f^*(T)=\bigcup_{m<\om}\hat{f}(T\re k_m)$ is a continuous map on its domain $\mathfrak{T}\re\tilde{T}$.
This map $f^*$ is equivalent to $f$ in the following sense:
For each $T\in\mathfrak{T}\re\tilde{T}$,
$f^*(T)=f([T])$.
In contrast, 
letting $\mathcal{C}=\{[T]:T\in\mathfrak{T}\re\tilde{T}\}$,
the map $f:\mathcal{C}\ra\mathcal{V}$ is not necessarily continuous.
However, $f\re\mathcal{C}$ is still represented by the  monotone finitary map $\hat{f}$; 
given $X\in\mathcal{C}$,
$f(X)=\bigcup_{m<\om}[[\hat{f}(\hat{X}\re k_m)]]$.


\section{Further connections between Tukey,  Rudin-Blass, and   Rudin-Keisler reductions}\label{sec.directapp}

In Lemma 9 of \cite{Raghavan/Todorcevic12},
Raghavan distilled properties of cofinal maps which, when satisfied, yield that Tukey reducibility above a q-point implies Rudin-Blass reducibility.
He then showed that continuous cofinal maps satisfy these properties, thus yielding  Theorem 10 in  \cite{Raghavan/Todorcevic12},
(see Theorem \ref{thm.RT} in Section \ref{sec.intro}).
The proof of the next theorem follows the general structure of Raghavan's proofs.
The key differences are  that we
start with a weaker assumption,
 basic
 maps on Fubini iterates of p-points, and obtain a finite-to-one map on the base tree $\hat{B}$ for the ultrafilter rather than the base $B$ itself.
While the following theorem may be of interest in itself, we will apply it to prove Theorems \ref{thm.RBgivesRK}
and \ref{thm.rkonG_k}
showing
 that for finite Fubini iterates of p-points, and for generic ultrafilters forced by $\mathcal{P}(\om^k)/\Fin^{\otimes k}$, Tukey reducibility above a q-point is equivalent to Rudin-Keisler reducibility.

\begin{thm}\label{thm.RBonBhat}
Suppose $\mathcal{U}$ is  a Fubini iterate of p-points and $\mathcal{V}$ is a q-point.
If
 $\mathcal{V}\le_T\mathcal{U}$,
then there is a finite-to-one function
$\tau:\tilde{T}\ra\om$,  for some $\tilde{T}\in\mathfrak{T}$, such that
$\{\tau[T]:T\in\mathfrak{T}\re\tilde{T}\}
\sse\mathcal{V}$.
\end{thm}

\begin{proof}
Let $B$ be the  front which is a base for $\mathcal{U}$, and as usual, let $\mathfrak{T}$ denote the set of all $\vec{\mathcal{U}}$-trees on $\hat{B}$.
We begin by establishing some useful notation:
Given $m<n$ and $T\sse[\om]^{<\om}$,
let $T\re [m,n)$ denote  
 the set of all $a\in T$ such that $m\le\max a<n$.

Let $f:\mathcal{U}\ra\mathcal{V}$ be a monotone cofinal map.
Let $\tilde{T}\in\mathfrak{T}$ be given by Theorem \ref{thm.allFubProd_p-point_cts},
so that $f$ is generated on $\mathfrak{T}\re\tilde{T}$ by some monotone basic map
$\hat{f}:\bigcup_{m<\om}2^{\tilde{T}\re k_m}\ra 2^{<\om}$.
Let $\psi:\mathcal{P}(\tilde{T})\ra\mathcal{P}(\om)$ be defined by
\begin{equation}
\psi(A)=\{k\in\om:\forall T\in\mathfrak{T}\re\tilde{T}\, (A\sse T\ra k\in f([T]))\},
\end{equation}
for each $A\sse \tilde{T}$.
Note that $\psi$ is monotone and further that for each $T\in\mathfrak{T}\re \tilde{T}$ and each $m<\om$,
$[[\hat{f}(T\re k_m)]]=\psi(T\re k_m)\cap m$.
By the same argument as in Lemma 8  in \cite{Raghavan/Todorcevic12},
we may assume that for each finite $A\sse\tilde{T}$, $\psi(A)$ is finite.

Notice  that for any  pair $m,j<\om$,
if $j\not\in\psi(\tilde{T}\re m)$, then there is a $T\in\mathfrak{T}\re\tilde{T}$ such that $T\sqsupset\tilde{T}\re m$ and $j\not\in f([T])$; hence, 
 $j\not\in [[\hat{f}(T\re k_j)]]$.
It follows that
for all $S\in\mathfrak{T}\re\tilde{T}$
satisfying $S\sqsupseteq T\re k_j$,
$j\not\in f([S])$  and hence also
$j\not\in \psi(S)$.
Without loss of generality, assume that $T\re[m,k_j)=\emptyset$;
 that is,
 each $a\in T$ has $\max a\not\in[m,k_j)$.
 Now if $A$ is a finite  $\sqsubset$-closed subset of $\tilde{T}$ and $A\re [m,k_j)=\emptyset$,
then there is an  $S\in\mathfrak{T}\re\tilde{T}$ such that
$\tilde{T}\re m\cup A\sse S$,
$S\re[m,k_j)=\emptyset$,
and $S\sqsupset T\re k_j$.
Then $j\not\in f([S])$,
so $j\not\in\psi(\tilde{T}\re m\cup A)$.

Define $g:\om\ra\om$ by $g(0)=0$; and  given $g(n)$, choose $g(n+1)>g(n)$ so that
\begin{equation}\label{eq1}
g(n+1)>\max\{k_{g(n)},\max(\psi(\tilde{T}\re k_{g(n)}))\}.
\end{equation}
Since $\mathcal{V}$ is a q-point, there is a $V_0\in\mathcal{V}$ such that for each $n<\om$,
$|V_0\cap[g(n),g(n+1))|=1$.
We may, without loss of generality, assume that $V_1=\bigcup_{n\in\om}[g(2n),g(2n+1))$ is in $\mathcal{V}$, and let $V=V_0\cap V_1$.
Enumerate $V$ as $\{v_i:i<\om\}$.
Notice that for each $i<\om$,
\begin{equation}\label{eq2}
g(2i)\le v_i<g(2i+1).
\end{equation}
Without loss of generality, assume that $v_0>0$.
Then $v_0\not\in\psi(\emptyset)$, since assuming $\mathcal{V}$ is nonprincipal, $\psi(\emptyset)$ must be empty.

Our construction ensures the following properties:  For all $i<\om$,
\begin{enumerate}
\item[(1)]
$g(i+1)>\max(k_{g(i)},\max(\psi(\tilde{T}\re k_{g(i)})))$;
\item[(2)]
$g(2i)\le v_i<g(2i+1)$;
\item[(3)]
$k_{v_i}<g(2i+2)$;
\item[(4)]
$\psi(\tilde{T}\re k_{v_i})\sse g(2i+2)$.
\end{enumerate}

We will now define  a strictly increasing function $h:\om\ra\om$ so that for each $i<\om$, the following hold:
\begin{enumerate}
\item[(a)]
$h(i)<k_{v_i}$;
\item[(b)]
$v_i\not\in\psi(\tilde{T}\re h(i))$;
\item[(c)]
For each finite, $\sqsubset$-closed set $A\sse\tilde{T}$,  if $v_i\in \psi(A)$  then $A\re[h(i),h(i+1))\ne\emptyset$.
\end{enumerate}

Define $h(0)=0$.
Then (a) - (c) trivially hold.
Suppose $h(i)$ has been defined so that (a) - (c) hold.
Define $h(i+1)=k_{v_i}$.
Then
$h(i+1)>h(i)$, since $k_{v_i}>h(i)$ by (a) of the induction hypothesis.
(a) holds, since $k_{v_i}<k_{v_{i+1}}$.
To see that (b) holds, note that
\begin{equation}
\psi(\tilde{T}\re h(i+1))=\psi(\tilde{T}\re k_{v_i})
\sse \psi(\tilde{T}\re k_{g(2i+1)})
\sse g(2i+2)
\le v_{i+1},
\end{equation}
where the inclusions hold by (3) and (1), and the inequality holds by (2).
Thus, $v_{i+1}\not\in\psi(\tilde{T}\re h(i+1))$.

To check (c), fix a finite $\sqsubset$-closed set $A\sse\tilde{T}$  such that $A\re [h(i),h(i+1))=\emptyset$; that is, for all $a\in A$, $\max a\not \in [h(i),h(i+1))$.
Let $A'=A\setminus (\tilde{T}\re h(i+1))$.
We claim that $v_i\not\in\psi(\tilde{T}\re h(i)\cup A')$.
By (b), $v_i\not\in\psi(\tilde{T}\re h(i))$.  Therefore,  $v_i\not\in [[\hat{f}(\tilde{T}\re h(i))]]$.
Now $(\tilde{T}\re h(i)\cup A')\re h(i+1)=\tilde{T}\re h(i)$, since $A'\re [h(i),h(i+1))=\emptyset$.
So $v_i\not\in [[\hat{f}((\tilde{T}\re h(i)\cup A')\re h(i+1))]]$.
Therefore, for each $S\in\mathfrak{T}\re\tilde{T}$ such that $S\re h(i+1)=(\tilde{T}\re h(i)\cup A')\re h(i+1)$, which we point out is the same as $\tilde{T}\re h(i)$,
we have $v_i\not\in f([S])$.
Thus, if $S\sqsupseteq \tilde{T}\re h(i)\cup A'$ and satisfies $S\re[h(i),h(i+1))=\emptyset$,
then $v_i\not\in f([S])$,
since this gets decided by height $k_{v_i}$ which equals $h(i+1)$.
Therefore, $v_i\not\in \psi(\tilde{T}\re h(i)\cup A')$, which proves (c).

Now we define a function $\tau:\tilde{T}\ra\om$ as follows: For $i<\om$ and $a\in\tilde{T}$, if $\max a\in [h(i),h(i+1))$, then
$\tau(a)=v_i$.

\begin{claim}
For each $T\in\mathfrak{T}\re\tilde{T}$, $\tau[T]\in\mathcal{V}$.
\end{claim}

\begin{proof}
Suppose not.  Then there is a $T\in\mathfrak{T}\re\tilde{T}$ such that $\tau[T]\not\in\mathcal{V}$; so $\tau[T]\in\mathcal{V}^*$.
Then there is an $S\in \mathfrak{T}\re T$ such that $f([S])\sse(\om\setminus \tau[T])\cap V\in\mathcal{V}$.
Let $j$ be least such that $\hat{f}(\chi_S\re k_j)\ne\emptyset$ and let $s=S\re k_j$.
Then $\emptyset\ne [[\hat{f}(\chi_S\re k_j)]]\sse\psi(s)$.
Fix some $v_i\in\psi(s)$.
Then $v_i\not\in \tau[T]$ since $v_i\in\psi(s)\sse f([S])\sse \om\setminus \tau[T]$.
However, $v_i\in\psi(s)$ implies $s\re [h(i),h(i+1))\ne\emptyset$, by (c).
For each $a\in s\re [h(i),h(i+1))$, $\tau(a)=v_i$.
Since $s\sse T$, we have $v_i\in\tau[T]$, a contradiction.
Thus, $\tau[T]\in\mathcal{V}$.
\end{proof}

Therefore, $\tau$ is a finite-to-one map from $\tilde{T}$ into $\om$,
 and the set of $\tau$-images of members of  $\mathfrak{T}\re\tilde{T}$ generate a filter contained inside $\mathcal{V}$.
\end{proof}

The previous theorem does not necessarily imply that the $\tau$-image of $\mathfrak{T}\re\tilde{T}$ generates $\mathcal{V}$.
However, under certain conditions, it does.
In the case that $\mathfrak{T}\re\tilde{T}$ generates an ultrafilter on base set the {\em tree} $\tilde{T}$, as is the case when all the p-points are the same selective ultrafilter, then the $\tau$-image of $\mathcal{U}$ {\em is} $\mathcal{V}$.
The following is a Rudin-Blass analogue, but on $\hat{B}$ instead of $B$.

\begin{cor}\label{cor.fubpowerselective}
If $\mathcal{U}$ is a Fubini power of some selective ultrafilter $\mathcal{V}$, where the base set for $\mathcal{U}$ is the front $B$,
then there is a finite-to-one map $\tau:\hat{B}\ra\om$ such that
$\{\tau[T]:T\in\mathfrak{T}\}$ generates $\mathcal{V}$.
\end{cor}

It is useful to point out the connection and contrast  between this corollary and the following previously known results.
Every Fubini power of some selective ultrafilter is Tukey equivalent to that selective ultrafilter (Corollary 37   in \cite{Dobrinen/Todorcevic11}).
Thus, if $\mathcal{U}$ is a Fubini iterate of one selective ultrafilter, then
 the q-point $\mathcal{V}$ in Theorem \ref{thm.RBonBhat} must be that  selective ultrafilter.
On the other hand,
 the only ultrafilters Tukey reducible to a selective ultrafilter are those ultrafilters isomorphic to some Fubini power of the selective ultrafilter (Theorem 24 of Todorcevic in \cite{Raghavan/Todorcevic12}).
Thus, the collection of ultrafilters Tukey equivalent to a given selective ultrafilter is exactly the collection of Fubini powers of that selective ultrafilter;
and any Fubini power of a selective ultrafilter 
is trivially Rudin-Keisler above 
 that same selective ultrafilter.

In
Corollary 56 of  \cite{Raghavan/Todorcevic12},
Raghavan showed that if $\mathcal{U}$ is some Fubini iterate of p-points and $\mathcal{V}$ is selective, then $\mathcal{V}\le_T\mathcal{U}$ implies $\mathcal{V}\le_{RK}\mathcal{U}$.
We now generalize this to q-points, though at the cost of  assuming $\mathcal{U}$ is only a finite Fubini iterate of p-points.

\begin{thm}\label{thm.RBgivesRK}
Suppose $\mathcal{U}$ is a finite iterate of Fubini products of p-points.
If $\mathcal{V}$ is a q-point and
 $\mathcal{V}\le_T\mathcal{U}$, then
$\mathcal{V}\le_{RK}\mathcal{U}$.
\end{thm}

\begin{proof}
Let $k$ denote the length of the Fubini iteration, so $[\om]^k$ is the  front which is a  base for $\mathcal{U}$.
Let $\tau$ be the finite-to-one map from Theorem \ref{thm.RBonBhat}, and without loss of generality, assume $\tau$ is defined on all of $[\om]^{\le k}$.
For each $T\in\mathfrak{T}$, we notice that $\bigcup_{1\le l\le k}\tau[T\cap [\om]^l]=\tau[T]\in\mathcal{V}$.
For each $T\in\mathfrak{T}$, let $L(T)=\{1\le l\le k:\tau[T\cap[\om]^l]\in\mathcal{V}\}$.
Then there is a $T\in\mathfrak{T}$ such that for all $S\in\mathfrak{T}\re T$, $L(S)=L(T)$.
Let $l=\max(L(T))$.

Now $\{S\cap [\om]^l:S\in\mathfrak{T}\re T\}$ generates an ultrafilter on base set $[\om]^l\cap T$;
further, for each $S\in\mathfrak{T}\re T$, $\tau[S\cap [\om]^l]$ is a member of $\mathcal{V}$  (since $l\in L(S)$).
Thus, $\{\tau[S\cap[\om]^l]:S\in\mathfrak{T}\re T\}$ generates an ultrafilter, and each of these $\tau$-images is in $\mathcal{V}$.
It follows that  $\{\tau[S\cap[\om]^l]:S\in\mathfrak{T}\re T\}$  generates $\mathcal{V}$.
If $l=k$, we are done, and in fact we have a Rudin-Blass map from $\mathcal{U}$ to $\mathcal{V}$.
If $l<k$, then define $\sigma:[\om]^k\ra \om$ by
$\sigma(a)=\tau(\pi_l(a))$.
Then $\sigma$ is a Rudin-Keisler map from $\mathcal{U}$ into $\mathcal{V}$.
\end{proof}

We point out that the {\em basic} maps on the generic ultrafilters $\mathcal{G}_k$ forced by $\mathcal{P}(\om^k)/\Fin^{\otimes k}$, $2\le k<\om$, in \cite{DobrinenJSL15}
have exactly the same properties as basic maps on $[\om]^k$  in this paper.
(See Definition 37 and Theorem 38 in \cite{DobrinenJSL15}.)
Hence, Theorem \ref{thm.RBonBhat} also applies to  these ultrafilters.
Since for each $1\le  l\le k$, the projection of $\mathcal{G}_k$ to $[\om]^l$ yields the generic ultrafilter $\mathcal{G}_l$,
 the same proof as in  Theorem \ref{thm.RBgivesRK}
yields the following theorem.

\begin{thm}\label{thm.rkonG_k}
Suppose $\mathcal{G}_k$ is a generic ultrafilter forced by $\mathcal{P}(\om^k)/\Fin^{\otimes k}$, for any $2\le k<\om$.
If $\mathcal{V}$ is a q-point and
 $\mathcal{V}\le_T\mathcal{G}_k$, then
$\mathcal{V}\le_{RK}\mathcal{G}_k$.
\end{thm}

\begin{rem}
We cannot in general weaken the requirement of q-point to rapid in Theorem \ref{thm.RBonBhat}.
 In \cite{Dobrinen/Todorcevic15}, it is shown that there are Tukey equivalent rapid p-points, and hence a Fubini iterate of such p-points Tukey equivalent to a rapid p-point,  which are Rudin-Keisler incomparable.
\end{rem}


\section{Ultrafilters Tukey reducible to Fubini iterates of p-points have
finitely generated Tukey reductions}\label{sec.finitegen}

In this section, we prove the analogue of Theorem \ref{thm.main}
for  the class of all ultrafilters which are Tukey reducible to some Fubini iterate of p-points.
Namely, in Theorem \ref{thm.FinitelyRepTukeyRed},
we prove that
every ultrafilter Tukey reducible to some Fubini iterate of p-points has finitely generated Tukey reductions (Definition \ref{defn.finitaryTred}).
This sharpens  a result of Raghavan (Lemma 16 in \cite{Raghavan/Todorcevic12}) by obtaining finitary maps which generate the original cofinal map on some filter base rather than some possibly different cofinal map.
Also, the class on which we obtain finitely generated Tukey reductions is closed under Tukey reduction, whereas the class where his result applies (basically generated ultrafilters) is not known to be closed under Tukey reduction.
 Theorem \ref{thm.FinitelyRepTukeyRed} allows us to extend Theorem 17 of Raghavan in  \cite{Raghavan/Todorcevic12}
 relating Tukey reduction to Rudin-Keisler reduction for basically generated ultrafilters to the class of all ultrafilters Tukey reducible to some Fubini iterate of p-points (see  Theorem \ref{thm.rextension} and the discussion preceding it).

The next lemma is the analogue of Lemma \ref{thm.PropACtsMaps} for the space $2^{\hat{B}}$ in place of $2^{\om}$.
As the proof is almost verbatim  by making the obvious changes, we omit it.

\begin{lem}[Extension lemma for  fronts]\label{lem.ExtCtsMaps}
Suppose $\mathcal{U}$ is a nonprincipal ultrafilter with  base set a  front $B$.
Suppose
$f:\mathcal{U}\ra\mathcal{V}$
is a  monotone cofinal map, and
there is a cofinal subset $\mathcal{C}\sse\mathcal{U}$  such that
  $f\re\mathcal{C}$ is 
  represented by a 
  monotone  basic function, in the sense of Definition \ref{defn.basicTredtree}.
Then there is a continuous, monotone $\tilde{f}:2^{\hat{B}}\ra 2^{\om}$
such that
\begin{enumerate}
\item
 $\tilde{f}$
is  represented by a monotone basic
 map
$\hat{f}:\bigcup_{m<\om}2^{\hat{B}\re k_m}\ra 2^{<\om}$, in the sense of Definition \ref{defn.basicTredtree}.
\end{enumerate}
Moreover,
defining the function $f':\mathcal{U}\ra 2^{\om}$ by
$f'(U)=\tilde{f}(\hat{U})$, for $U\in\mathcal{U}$,
\begin{enumerate}
 \item[(2)]
$f'\re\mathcal{C}=f\re\mathcal{C}$; and
\item[(3)]
$f'\re\mathcal{U}$ is a cofinal map from $\mathcal{U}$ to $\mathcal{V}$.
\end{enumerate}
\end{lem}

\begin{defin}[Finitely generated Tukey reduction]\label{defn.finitaryTred}
We say that an ultrafilter $\mathcal{V}$ on  base set  $\om$ has {\em finitely generated Tukey reductions} if
for each monotone cofinal map $f:\mathcal{V}\ra\mathcal{W}$,
there is a cofinal subset $\mathcal{C}\sse\mathcal{V}$, a strictly increasing sequence $(k_m)_{m<\om}$, and a function
$\hat{f}:C\ra 2^{<\om}$, where $C=\{X\re k_m:X\in\mathcal{C},\ m<\om\}$, such that
\begin{enumerate}
\item
$\hat{f}$ is level  preserving:
For each $m<\om$ and $s\in C$, $|s|=k_m$ implies $|\hat{f}(s)|=m$;
\item
$\hat{f}$ {\em  generates} $f$ on $\mathcal{C}$: For each $X\in\mathcal{C}$,
\end{enumerate}
\begin{equation}
f(X)=\bigcup_{m<\om}[[\hat{f}(X\re k_m)]].
\end{equation}
\end{defin}

The difference between a basic Tukey reduction and a finitely generated Tukey reduction is that the map $\hat{f}$ in Definition \ref{defn.finitaryTred} is not required to be end-extension preserving.

Now we prove the main theorem of this section.
This is the analogue of  Theorem \ref{thm.PropB} (which holds for ultrafilters Tukey reducible to some p-point)
to the setting of all ultrafilters Tukey reducible to some  Fubini iterate of p-points.

\begin{thm}\label{thm.FinitelyRepTukeyRed}
Let $\mathcal{U}$ be any Fubini iterate of  p-points.
If $\mathcal{V}\le_T\mathcal{U}$, then
$\mathcal{V}$  has finitely generated Tukey reductions.
\end{thm}

\begin{proof}
Suppose that  $\mathcal{U}$  is an iteration of Fubini products of p-points and  that $\mathcal{V}\le_T\mathcal{U}$.
Let $B$ be a front which is a base for $\mathcal{U}$, and 
without loss of generality, assume that $\om$ is the base set for the ultrafilter $\mathcal{V}$.
By Theorem \ref{thm.allFubProd_p-point_cts},
$\mathcal{U}$ has basic Tukey reductions.
Applying   Lemma \ref{lem.ExtCtsMaps},
we obtain a  continuous  monotone map $\tilde{f}:2^{\hat{B}}\ra 2^{\om}$
which is
generated by a monotone basic map $\hat{f}:\bigcup_{m<\om}2^{\hat{B}\re k_m}\ra 2^{<\om}$, for some increasing sequence $(k_m)_{m<\om}$.
Hence,
 for each $A\sse\hat{B}$, $\tilde{f}(A)=\bigcup_{m<\om}[[\hat{f}(A\re k_m)]]$.
Furthermore,
defining $f(U)=\tilde{f}(\hat{U})$ for $U\in\mathcal{U}$, we see that
 $f:\mathcal{U}\ra\mathcal{V}$ is a monotone cofinal map.

Suppose $\mathcal{W}\le_T\mathcal{V}$, and let
$h:\mathcal{V}\ra\mathcal{W}$ be a monotone cofinal map.
Extend $h$ to the map $\tilde{h}:2^{\om}\ra 2^{\om}$  defined as follows:
For each $X\in 2^{\om}$, let
\begin{equation}
\tilde{h}(X)=\bigcap\{h(V):V\in\mathcal{V}\mathrm{\ and\ }V\contains X\}.
\end{equation}
It follows from $h$ being monotone that  $\tilde{h}$ is monotone and that $\tilde{h}\re\mathcal{V}=h$.

Letting $\tilde{g}$ denote $\tilde{h}\circ \tilde{f}$,
we see that the map
 $\tilde{g}: 2^{\hat{B}}\ra 2^{\om}$ is monotone.
For $U\in\mathcal{U}$,
$\tilde{g}(\hat{U})=\tilde{h}(\tilde{f}(\hat{U}))=\tilde{h}(f(U))=h\circ f(U)$.
Thus, letting
$g$ denote $h\circ f$, we see that
 $g:\mathcal{U}\ra\mathcal{W}$ is a monotone cofinal map with the property that
 for each $U\in \mathcal{U}$,
$g(U)=\tilde{g}(\hat{U})$.
By Theorem \ref{thm.allFubProd_p-point_cts},
there is a $\vec{\mathcal{U}}$-tree $\tilde{T}$ and an increasing sequence $(k_m)_{m<\om}$  such that
the restriction of $g$ to 
$\mathcal{C}=\{[T]:T\in \mathfrak{T}\re \tilde{T}\}$ is
generated by some
monotone basic  map
$\hat{g}:C\ra 2^{<\om}$,
where
$$
C=\{T\re k_m: T\in\mathfrak{T}\re \tilde{T}\mathrm{\ and \ } m<\om\}.
$$
Notice that $C=\bigcup_{m<\om}2^{\tilde{T}\re k_m}$.
Without loss of generality, we may assume that the levels $k_m$ are the same for $\hat{f}$ and $\hat{g}$, by taking the minimum of the two $m$-th levels.

Let  $g^*$
be the function on $2^{\tilde{T}}$ into $2^{\om}$ determined by $\hat{g}$ as follows:
For $A \in 2^{\tilde{T}}$,
define
\begin{equation}
g^*(A)=\bigcup_{m<\om}\hat{g}(A\re k_m).
\end{equation}
Since $\hat{g}$ is end-extension preserving, it follows  that for each $A\in 2^{\tilde{T}}$ and  $m<\om$,
$g^*(A)\re m=\hat{g}(A\re k_m)$.
We point out that 
the restriction of $g^*$ to 
$\mathfrak{T}\re\tilde{T}$ yields exactly $\tilde{g}$, since $\hat{g}$ generates $\tilde{g}$ on 
$\mathfrak{T}\re\tilde{T}$.

\begin{claim1}
For each $A \in  2^{\tilde{T}}$,
\begin{equation}\label{eq.star=intersection}
g^*(A)=\bigcap\{g(X):X\in\mathcal{C}\mathrm{\ and \ }\hat{X}\contains A\}\contains \tilde{g}(A).
\end{equation}
\end{claim1}

\begin{proof}
Let $A$ be any member of $2^{\tilde{T}}$.
Let $m$ be given and
note that
  $A\re k_m\in C$.
For any $X\in\mathcal{C}$ such that $\hat{X}\re k_m= A\re k_m$, we have
 \begin{equation}
\hat{g}(A\re k_m) = \hat{g}(\hat{X}\re k_m)=
\tilde{g}(\hat{X})\re m=g(X)\re m.
\end{equation}
Since $\mathcal{C}=\{[T]:T\in\mathfrak{T}\re\tilde{T}\}$,
there is  an $X\in\mathcal{C}$ such that $\hat{X}\contains A$ and $\hat{X}\re k_m= A\re k_m$.
(This is the key property of $\mathcal{C}$ needed for this proof.)
Thus,
\begin{equation}\label{eq.56}
[[\hat{g}(A\re k_m)]]
= Y\re m
\contains \tilde{g}(A)\re m,
\end{equation}
where $Y=
\bigcap\{g(X):X\in\mathcal{C}\mathrm{\  and\ }\hat{X}\contains A\}$.
Taking the union over all  $m<\om$ in (\ref{eq.56}) yields the claim.
\end{proof}

Define \begin{equation}
D=\{\hat{f}(s):s\in C\}
\textrm{\ and  \ } \mathcal{D}=f[\mathcal{C}].
\end{equation}
 Then $\mathcal{D}$ is cofinal in $\mathcal{V}$, and every member of $\mathcal{D}$  is a limit of members of $D$.
We
point out that   $\mathfrak{T}\re\tilde{T}$ is a subspace of $2^{\hat{B}}$,
and
 the closure of  $\mathfrak{T}\re\tilde{T}$ in  $2^{\hat{B}}$ is the compact space $2^{\tilde{T}}$.
Define a function $\hat{h}:D\ra 2^{<\om}$ as follows:
For $t\in D\cap 2^m$ and $i\in m$,  define
\begin{equation}
\hat{h}(t)(i)=\min\{\hat{g}(s)(i):s\in C\cap 2^{\hat{B}\re k_m}\mathrm{\ and\ }\hat{f}(s)=t\}.
\end{equation}
That is,
 $\hat{h}(t)$ is the function from $m$ into $2$  such that
for $i\in m$,
$\hat{h}(t)(i)= 1 $ if and only if
$\hat{g}(s)(i)=1$ for all
$s\in C\cap 2^{\hat{B}\re k_m}$ satisfying $\hat{f}(s)=t$.
By definition, $\hat{h}$ is level preserving.

We proceed to prove that
$\hat{h}$ represents $h$ on $\mathcal{D}$.
Fix $Y\in\mathcal{D}$.
Then there is an $X\in\mathcal{C}$ such that $f(X)=Y$.
For each $m<\om$,
 $\hat{f}(\hat{X}\re k_m)=Y\re m$,
so
\begin{align}
[[\hat{h}(Y\re m)]]&\sse [[\hat{g}(\hat{X}\re k_m)]]
=\tilde{g}(\hat{X})\re m \cr
&=g(X)\re m 
= h\circ f (X)\re m=h(Y)\re m.  \cr
\end{align}
Thus, $\bigcup_{m<\om}[[\hat{h}(Y\re m)]]\sse h(Y)$.

Next we show that for each $l\in h(Y)$, there is some $n$ such that $l\in [[\hat{h}(Y\re n)]]$.
Let $m=l+1$ and let $t=Y\re m$ .
Let
\begin{equation}
S_l=\{s\in C\cap 2^{\hat{B}\re k_m}:  \forall A\in 2^{\tilde{T}}\mathrm{\ with \ }
 s\sqsubset A,\  \tilde{f}(A)\ne Y\}.
\end{equation}

\begin{claim2}
For each $s\in S_l$,
there is an $n_s$ such that each $s'\in 2^{\hat{B}\re n_s}$ with 
$s'\sqsupset s$ 
satisfies
 $\hat{f}(s')\not\sqsubset Y$.
 \end{claim2}

\begin{proof}
If not, then for some $s\in S_l$, for each $n$ there is an $s'\in C$ of length $n$ with $s'\sqsupseteq s$
and some $A\in 2^{\tilde{T}}$ with
$A\sqsupset s'$ 
such that 
 $\tilde{f}(A)=Y$.
By K\"{o}nig's Lemma, there is a sequence $s\sqsubset s_0\sqsubset s_1\sqsubset\dots$ of strictly increasing lengths such that
for each $i<\om$, there is some $A_i\in 2^{\tilde{T}}$ such that $A_i\sqsupset s_i$ and $\tilde{f}(A_i)=Y$.
Letting $A'=\bigcup_{i<\om}s_i$,
we see that $A'\in 2^{\tilde{T}}$ and that $\tilde{f}(A')=Y$, by continuity of $f$ and the fact that $\hat{f}$ generates $\tilde{f}$ on $2^{\tilde{T}}$.
But this contradicts the fact that
 $s\in S_l$.
\end{proof}

Since $S_l$ is finite, we may take $n=\max\{n_s:s\in S_l\}$.
Then for all $s'$ of length $k_n$, if $\hat{f}(s')\sqsubset Y$
then $s'$ does not end-extend any $s$ in $S_l$.

\begin{claim3}
For all $s'\in C$ of length $k_n$ such that $\hat{f}(s')\sqsubset Y$,
 $l\in [[\hat{g}(s')]]$.
 \end{claim3}

\begin{proof} 
For each $s'\in C\cap 2^{\hat{B}\re k_n}$
satisfying $\hat{f}(s')= Y\re n$, we see that
$s'\re k_m$ is not in $S_l$.
So there is some $A\in 2^{\tilde{T}}$ such that $A\sqsupset s'\re k_m$ and $\tilde{f}(A)=Y$.
It follows that
\begin{equation}
[[\hat{g}(s')\re m]]=[[\hat{g}(s'\re k_m)]]\contains \tilde{g}(A)\re m=\tilde{h}\circ\tilde{f}(A)\re m=h(Y)\re m,
\end{equation}
where the $\contains$ follows from Claim 1,
since $[[\hat{g}(s'\re k_m)]]=g^*(A)\re m \contains \tilde{g}(A)\re m$.
Thus,  $l\in [[\hat{g}(s')]]$, for each $s'\in 2^{\tilde{T}\re k_n}$
satisfying $\hat{f}(s')= Y\re n$.
\end{proof}

Therefore, $l\in [[\hat{h}(Y\re n)]]$.
Thus, for each $l\in h(Y)$, there is an $n_l$ such that $l\in [[\hat{h}(Y\re n_l)]]$.
It follows that, for any $j<\om$,
there is an $n$ such that $\hat{h}(\chi_Y\re n)\re j=h(Y)\re j$.
This $n$ may be obtained by
taking the maximum of  the $n_s$ over all $s\in \bigcup\{S_l:l<j\}$,
Hence, $\bigcup_{m<\om}[[\hat{h}(Y\re m)]]=h(Y)$.

Thus,  $h\re \mathcal{D}$ is finitely represented by $\hat{h}$ on $D$.
\end{proof}

Theorem \ref{thm.FinitelyRepTukeyRed} is now applied to extend
 Theorem 17 of Raghavan in \cite{Raghavan/Todorcevic12}
to all ultrafilters Tukey reducible to some Fubini iterate of p-points.
Raghavan showed
 that for any basically generated ultrafilter $\mathcal{U}$, whenever $\mathcal{V}\le_T\mathcal{U}$  there is a filter $\mathcal{U}(P)$ which is Tukey equivalent to $\mathcal{U}$ such that $\mathcal{V}\le_{RK}\mathcal{U}(P)$.
It is routine to check that the maps in
 Theorem \ref{thm.FinitelyRepTukeyRed}
satisfy the conditions of the maps in
Theorem 17 of Raghavan in \cite{Raghavan/Todorcevic12}.
Thus, we obtain the following.

\begin{thm}\label{thm.rextension}
If $\mathcal{U}$ is Tukey reducible to a  Fubini iterate of  p-points, then
for each $\mathcal{V}\le_T\mathcal{U}$, there is a filter $\mathcal{U}(P)\equiv_T\mathcal{U}$ such that $\mathcal{V}\le_{RK}\mathcal{U}(P)$.
\end{thm}

Here, assuming without loss of generality that the base set for $\mathcal{U}$ is $\om$,
$P$ is the collection of $\sqsubset$-minimal finite subsets $s$ of $\om$  for which  $\hat{h}(s)\ne\emptyset$, where $\hat{h}$ witnesses that a given monotone cofinal $h:\mathcal{U}\ra\mathcal{V}$ is finitely generated.
$\mathcal{U}(P)$ is the collection of all sets of the form $\{s\in P:s\sse U\}$, for $U\in\mathcal{U}$.

\begin{rem}
The  same proof of Theorem \ref{thm.FinitelyRepTukeyRed}
works for the
basic cofinal maps for the generic ultrafilters $\mathcal{G}_k$ forced by $\mathcal{P}(\om^k)/\Fin^{\otimes k}$, $2\le k<\om$, in \cite{DobrinenJSL15}.
Thus, Theorem
 \ref{thm.rextension}
also  holds when $\mathcal{U}$ is an ultrafilter forced by $\mathcal{P}(\om^k)/\Fin^{\otimes k}$.
\end{rem}


\section{Open problems}\label{sec.op}

We conclude this paper by highlighting  some of the more important open  problems in this area.
 Theorem \ref{thm.main} showed
 that every ultrafilter Tukey below a p-point has  continuous Tukey reductions.

\begin{problem}
Determine the class of all ultrafilters that have continuous Tukey reductions.
\end{problem}

In particular, are  there ultrafilters not Tukey reducible to a p-point which
 satisfy the conditions of Theorem \ref{thm.PropB}?

By Theorem  56 of Dobrinen and 
Trujillo\footnote[2]{Attributions of the results  of the various authors  are made clear in \cite{Dobrinen/Mijares/Trujillo14}.} in \cite{Dobrinen/Mijares/Trujillo14}, under very mild conditions,
any ultrafilter selective for some topological Ramsey space has  basic, an hence, continuous (with respect to  its metric topology) Tukey reductions.
This is especially of interest when the ultrafilter associated with the topological Ramsey space is not a p-point.
It should be the case that by arguments similar to those in this paper one can prove the following.

\begin{problem}
Prove the analogues of Theorems  \ref{thm.PropB}, \ref{thm.allFubProd_p-point_cts}, and \ref{thm.FinitelyRepTukeyRed}  for stable ordered union ultrafilters and their iterated Fubini products, and more generally for ultrafilters selective for some topological Ramsey space, with respect to the correct topologies.
\end{problem}

More generally, we would like to know the following.

\begin{problem}
Determine the class of all ultrafilters which have finitely generated Tukey reductions.
Is this the same as the class of all ultrafilters with Tukey type strictly below the maximum Tukey type?
\end{problem}

In Section \ref{sec.directapp}, we applied Theorem  \ref{thm.RBonBhat}
to obtain  more examples when Tukey reducibility  implies Rudin-Keisler reducibility.
 Theorem \ref{thm.RBgivesRK}  improves on one aspect of
Corollary 56  in\cite{Raghavan/Todorcevic12} of Raghavan  provided that there are q-points which are not selective and which are Tukey below some Fubini iterate of p-points.
Do such ultrafilters ever exist?

\begin{problem}
Is there a q-point which is not selective which is Tukey reducible to some finite Fubini iterate of p-points?
Or does $\mathcal{V}\le_T\mathcal{U}$ with $\mathcal{U}$ a Fubini iterate of p-points and $\mathcal{V}$ a q-point imply that $\mathcal{V}$ is actually selective?
\end{problem}

\begin{problem}
Can Theorem  \ref{thm.RBgivesRK} be extended to all countable iterates of Fubini products of p-points?
Are  similar results true  for all ultrafilters Tukey reducible to some Fubini iterate of p-points?
\end{problem}

Question 25 in  \cite{Dobrinen/Todorcevic11}
asks whether asks whether every ultrafilter Tukey reducible to a p-point is basically generated.
Question 26 in \cite{Dobrinen/Todorcevic11}
asks whether
 the classes of basically generated and Fubini  iterates of p-points the same, or whether the former is  strictly larger than the latter?
Though these questions in general are still open, we ask the even more general questions.

\begin{problem}\label{prob.bgTdownwardinherit}
Is the property of being basically generated inherited under Tukey reduction?
That is, if $\mathcal{V}$ is Tukey reducible to a basically generated ultrafilter, is $\mathcal{V}$  necessarily basically generated?

Or the possibly weaker problem:
 If $\mathcal{V}$ is Tukey reducible to some Fubini iterate of p-points, is $\mathcal{V}$  necessarily basically generated?
\end{problem}

There are certain collections of p-points, in particular those associated with the  topological Ramsey spaces in  \cite{Dobrinen/Todorcevic14}, \cite{Dobrinen/Todorcevic15}, and  \cite{Dobrinen/Mijares/Trujillo14},
for which every ultrafilter Tukey below some Fubini iterate of these p-points is again a Fubini iterate of these p-points and hence basically generated.
However, the above questions are in general still open.

Work in this paper and work in \cite{Raghavan/Todorcevic12} found conditions when Tukey reducibility implies Rudin-Keisler or even Rudin-Blass reducibility.

\begin{problem} When in general does $\mathcal{U}\ge_T\mathcal{V}$ imply $\mathcal{U}\ge_{RK}\mathcal{V}$ or $\mathcal{U}\ge_{RB}\mathcal{V}$?
\end{problem}


Finally, how closely related are the properties of having finitely generated Tukey reductions and having Tukey type below the maximum?

\begin{problem}
Does $(\mathcal{U},\contains)<_T ([\om]^{<\om}),\sse)$
imply that $\mathcal{U}$ has finitely generated Tukey reductions?
\end{problem}



\bibliographystyle{amsplain}
\bibliography{references}

\end{document}